\documentclass{amsart}
\usepackage[utf8]{inputenc}
\usepackage{amsfonts}
\usepackage{amsmath}
\numberwithin{equation}{section}
\usepackage{amssymb}
\usepackage{amsthm}
\usepackage{mathrsfs}
\usepackage{mathtools}
\usepackage{booktabs,array,tabularx}
\usepackage{placeins,needspace}
\usepackage{graphicx}
\usepackage{xcolor}
\usepackage{tikz}
\usepackage{pgfplots}
\pgfplotsset{compat=1.18}
\usepackage[hidelinks]{hyperref}
\usepackage[capitalise]{cleveref}
\usepackage[top=3cm,bottom=3cm,left=2.5cm,right=2.5cm,headsep=10pt,letterpaper]{geometry} 

\usepackage[
    backend=biber,
    style=alphabetic,
    sorting=nyt,
    maxbibnames=99
]{biblatex}
\theoremstyle{plain}
\newtheorem{theorem}{Theorem}[section]
\newtheorem{proposition}[theorem]{Proposition}
\newtheorem{lemma}[theorem]{Lemma}
\newtheorem{corollary}[theorem]{Corollary}

\theoremstyle{definition}

\newtheorem{example}[theorem]{Example}

\theoremstyle{remark}
\newtheorem{remark}[theorem]{Remark}

\allowdisplaybreaks[2]
\newcommand{\R}{\mathbb R}
\newcommand{\Z}{\mathbb Z}
\newcommand{\T}{\mathbb T}
\newcommand{\E}{\mathbb E}

\newcommand{\Var}{\operatorname{Var}}
\newcommand{\dd}{\,\mathrm d}

\DeclareMathOperator{\diag}{diag}

\title[Fisher completions and regularization]{Fisher Completions and Statistical Regularization of Wrapped Gaussian Families}

\author[G. Cho]{Gunhee Cho}
\address[G. Cho]{Department of Mathematics \\ Texas State University \\ San Marcos, TX 78666, USA}
\email{wvx17@txstate.edu}

\author[H. C. Jang]{Hyun Chul Jang}
\address[H. C. Jang]{Department of Mathematics \\ Texas State University \\ San Marcos, TX 78666, USA}
\email{hcjang@txstate.edu}
\thanks{Corresponding author: Hyun Chul Jang.}

\subjclass[2020]{Primary 53B12; Secondary 62B11, 62H11, 62C20, 53C23, 11H06, 60E05}

\keywords{Fisher information, wrapped normal distribution, directional statistics, metric completion, covariance regularization, local compactness, Assouad dimension, minimax risk}

\hypersetup{pdftitle={Fisher Completions and Statistical Regularization of Wrapped Gaussian Families},pdfauthor={Gunhee Cho and Hyun Chul Jang}}

\begin{document}
\raggedbottom
\begin{abstract}
We characterize the intrinsic Fisher completion of the full-covariance wrapped Gaussian family on a fixed torus. Its points are indexed by primitive frequency subgroups, characters, and positive quadratic forms; the completion topology is total variation. Every added density is positive and analytic, yet in dimension at least two every neighborhood of every added point is noncompact and has infinite packing at a fixed relative scale. This gives local minimax risk of order the neighborhood radius at every sample size. A fixed covariance floor $\Sigma\succeq\varepsilon I$, with no upper bound, restores compactness and uniform estimation: the $L^1$ risk is at most $C_{\varepsilon,d}n^{-1/2}\{\log(n+1)\}^{d/4}$, and model-valued estimators are uniformly consistent in expected Fisher distance. The completed geometry also admits explicit minimizing geodesics through every added law, including pairs of finite-parameter laws whose distance is not attained inside the finite family. For compatible distinct pieces we obtain a sharp right-angle distance asymptotic. Heat contractions and canonical retractions give contractible small balls despite the failure of local compactness.
\end{abstract}
\maketitle
\setcounter{tocdepth}{1}\tableofcontents
\section{Introduction}
When a Gaussian observation is recorded only modulo a fixed period lattice, its covariance can diverge while the observed law remains at finite Fisher distance. What points complete the statistical model, and what does the resulting geometry imply for estimation? We answer these questions for unrestricted covariance on the fixed torus $\T^d=\R^d/(2\pi\Z^d)$. The central distinction is between the full family and the same family with a fixed positive lower bound on covariance. The full completion has infinitely many separated laws in every neighborhood of every added point. A covariance floor restores compactness and uniform estimation without imposing any upper covariance bound.

Let $\mathcal P_d$ be the positive-definite covariance matrices. A point $P=(\mu,\Sigma)\in\mathcal M_d=\T^d\times\mathcal P_d$ is the law of the angles $Y=X\bmod2\pi$, where $X\sim N_d(\mu,\Sigma)$. Its density $F_P$ relative to normalized Haar measure $\nu$ is determined by its Fourier moments
\[
 m_n(F)=\int e^{in\cdot\theta}F(\theta)\dd\nu(\theta)=\E\,e^{in\cdot Y},
 \qquad
 m_n(F_P)=e^{in\cdot\mu-n^\top\Sigma n/2}\qquad(n\in\Z^d).
\]
The angle $n\cdot Y=n_1Y_1+\cdots+n_dY_d$ is well defined modulo $2\pi$ because $n$ is an integer vector, and $|m_n|$ measures how concentrated it is. Let $d_{\mathrm W}$ be the intrinsic Fisher distance of $\mathcal M_d$, $\widehat{\mathcal M}_{d,\mathrm W}$ its metric completion, and call the points of $\widehat{\mathcal M}_{d,\mathrm W}\setminus\mathcal M_d$ added points. Write $U_d$ for the uniform law and $\mathrm{WN}_r(\beta,B)$ for the $r$-dimensional wrapped normal law with mean $\beta$ and covariance $B$. A subgroup $L\le\Z^d$ is primitive if $kn\in L$ implies $n\in L$ for every nonzero integer $k$. We use $\|F-G\|_1=2d_{\mathrm{TV}}(F,G)$. The period lattice $2\pi\Z^d$ is no restriction: observations modulo any fixed full-rank lattice reduce to this case by a linear change of variables, with the dual lattice as frequency lattice (Remark~\ref{rem:periodlattice}).

\begin{theorem}[Intrinsic completion and its local geometry]\label{thm:maincompletion}
Let $d\ge1$. The following hold.
\begin{enumerate}
\item \emph{Classification.} The density map extends to a homeomorphism from $\widehat{\mathcal M}_{d,\mathrm W}$ onto the total-variation closure of $\{F_P:P\in\mathcal M_d\}$. A density $F$ in this closure determines
\[
 L=\{n\in\Z^d:m_n(F)\neq0\},\qquad
 \chi(n)=\frac{m_n(F)}{|m_n(F)|},\qquad
 Q(n)=-2\log|m_n(F)|\quad(n\in L),
\]
where $L$ is a primitive subgroup, $\chi:L\to S^1$ is a character, and $Q$ extends to a positive-definite quadratic form on $\operatorname{span}_{\R}L$. Conversely, every such triple $(L,\chi,Q)$ arises, and it determines $F$ through
\[
 m_n(F)=\begin{cases}\chi(n)e^{-Q(n)/2},&n\in L,\\0,&n\notin L.\end{cases}
\]
With $r=\operatorname{rank}L$, in integer torus coordinates whose last $r$ basis vectors span $L$, the law is $U_{d-r}\otimes\mathrm{WN}_r(\beta,B)$, where $\beta\in\T^r$ gives the values of $\chi$ on that basis and $B$ is the Gram matrix of $Q$. The added points are exactly the laws with $r<d$. All these densities are positive and real analytic.
\item \emph{Global regularity.} The completion and its added set contract to $U_d$, and every point has a neighborhood basis of contractible metric balls. Their Hausdorff dimensions are respectively $d(d+3)/2$ and $(d-1)(d+2)/2$.
\item \emph{Local failure.} If $d\ge2$, every sufficiently small ball $B(E,R)$ about an added point $E$ contains infinitely many points separated by at least $c_E R$, with $c_E>0$ independent of $R$. Hence no added point has a compact neighborhood, and every neighborhood of one has infinite Assouad dimension.
\end{enumerate}
\end{theorem}

Concretely, a finite-parameter sequence $(\mu_j,\Sigma_j)$ converges to the point $(L,\chi,Q)$ exactly when
\[
 n^\top\Sigma_jn\to Q(n)\ \text{and}\ e^{in\cdot\mu_j}\to\chi(n)\quad(n\in L),
 \qquad
 n^\top\Sigma_jn\to\infty\quad(n\notin L).
\]
Thus $L$ consists of the angles $n\cdot Y$ that remain nonuniform in the limit, and $\chi(n)$ and $Q(n)$ are the limiting mean direction and variance of that angle: under the limit law $n\cdot Y\sim\mathrm{WN}_1(\arg\chi(n),Q(n))$. Example~\ref{ex:d2} and Figure~\ref{fig:latticeexample} work this out for $d=2$, and Theorem~\ref{thm:completionclassification} gives the equivalent criterion through the lattice minimum of a Schur complement.

The completion statement is stronger than a Fourier description of limiting laws: each limit must be reached by finite-length paths inside the wrapped family, and different approaches to the same law must define the same metric endpoint. The main technical estimate controls this intrinsic approach uniformly over the cross-covariance blocks. In adapted coordinates, the lattice minimum of the Schur complement gives a necessary and sufficient convergence criterion (Theorem~\ref{thm:completionclassification}). Along each fixed positive-definite covariance ray, the distance has an exact leading coefficient determined by the shortest lattice shell (Proposition~\ref{prop:sharpendpoint}). Ambient Fisher--Rao or Hellinger distances alone do not give these intrinsic conclusions, since their paths may leave the model.

The geometric obstruction has a statistical counterpart, and a simple restriction removes both. In the next statement $Y_1,\dots,Y_n$ are independent observations from $P$. An estimator is a measurable map $\widehat F_n$ from $(\T^d)^n$ to the probability densities on $\T^d$, taken modulo $\nu$-null sets with the $L^1$ metric; minimax risks range over all such estimators, including those whose values lie outside the model. Balls $B(E,R)$ are taken in the completed distance $d_{\mathrm W}$, and distances on restricted families are inherited from the full completion. For $\varepsilon>0$ put
\[
 \mathcal M_d(\varepsilon)=\{(\mu,\Sigma)\in\mathcal M_d:\Sigma\succeq\varepsilon I_d\},
\]
a lower bound on every covariance eigenvalue in the fixed observation coordinates, with no upper bound.

\begin{theorem}[Unrestricted obstruction and spectral regularization]\label{thm:mainstat}
Let $d\ge1$ and $\varepsilon>0$. The following hold.
\begin{enumerate}
\item \emph{Obstruction near added points.} Suppose $d\ge2$ and let $E$ be an added point. There are constants $R_E,\lambda_E>0$ such that for all $0<R<R_E$ and all sample sizes $n\ge1$,
\[
 \lambda_E R\le
 \inf_{\widehat F_n}\sup_{P\in B(E,R)\cap\mathcal M_d}
 \E_{P^{\otimes n}}\|\widehat F_n-F_P\|_1\le R.
\]
The same bounds hold for completed-Fisher loss and estimators with values in $\widehat{\mathcal M}_{d,\mathrm W}$. In particular, no estimator is uniformly consistent on any fixed neighborhood of an added point.
\item \emph{Exact compactness criterion.} A set $\mathcal F\subset\mathcal M_d$ has compact closure in $d_{\mathrm W}$ if and only if its Fourier moments decay uniformly:
\[
 \sup_{P\in\mathcal F}|m_n(F_P)|\longrightarrow0
 \quad\text{equivalently}\quad
 \inf_{(\mu,\Sigma)\in\mathcal F}n^\top\Sigma n\longrightarrow\infty
 \qquad(|n|\to\infty).
\]
Every such $\mathcal F$ admits estimators with values in $\mathcal F$ that are uniformly consistent on $\mathcal F$, both in $L^1$ and in completed-Fisher distance. A spectral floor is a sufficient condition: $\mathcal M_d(\varepsilon)$ satisfies the criterion, and its closure consists exactly of the points $(L,\chi,Q)$ with $Q(v)\ge\varepsilon|v|^2$ on $\operatorname{span}_{\R}L$. It is not necessary: the sequence of Example~\ref{ex:d2} with $\lambda_{\min}(\Sigma_j)\to0$ converges to $U_2$.
\item \emph{Uniform estimation under a spectral floor.} There are constants $c_\varepsilon,C_{\varepsilon,d}>0$ with
\[
 \frac{c_\varepsilon}{\sqrt n}\le
 \inf_{\widehat F_n}\sup_{P\in\mathcal M_d(\varepsilon)}
 \E_{P^{\otimes n}}\|\widehat F_n-F_P\|_1
 \le C_{\varepsilon,d}\frac{\{\log(n+1)\}^{d/4}}{\sqrt n}.
\]
There are estimators taking values in $\mathcal M_d(\varepsilon)$ that attain the upper order and are uniformly consistent in expected completed-Fisher distance.
\end{enumerate}
\end{theorem}

Part~(1) is Corollary~\ref{cor:nouniformestimation}. Part~(2) is Proposition~\ref{prop:compactcriterion} and Corollary~\ref{cor:compactconsistency}, together with Theorem~\ref{thm:spectralfloor}(1) for the spectral-floor closure. Part~(3) is Theorem~\ref{thm:spectralfloor}(2) and Proposition~\ref{prop:spectralfloorlower}, with $c_\varepsilon=3e^{-\varepsilon/2}\min\{1,\sqrt\varepsilon\}/(8\pi)$. The logarithmic gap in part~(3) is not claimed to be optimal. Two remarks clarify the scope of part~(1). First, a fixed Fisher ball about an added point is large in parameter terms: by Theorem~\ref{thm:completionclassification} it contains parameters with arbitrary means and arbitrarily large condition numbers (Section~\ref{sec:statistics}). Second, the obstruction concerns each fixed neighborhood; an estimator that always outputs a known center $E$ already controls the error on balls whose radii shrink to zero. Part~(2) explains the contrast exactly. The separated sequences of Theorem~\ref{thm:maincompletion}(3) keep a moment of fixed size at frequencies $n$ with $|n|\to\infty$, which is precisely what uniform decay forbids; a spectral floor forbids it through the envelope $|m_n|\le e^{-\varepsilon|n|^2/2}$, but so do other families whose smallest eigenvalue tends to zero.

The mechanism can be seen in two dimensions. Fix $a>0$, put $q=e^{-a/2}$, and consider
\[
 \Sigma_k=a\begin{pmatrix}1+k^2&-k\\-k&1\end{pmatrix},\qquad k\ge0.
\]
Integer shears preserve the wrapped Fisher metric and uniformity. All these laws therefore lie at the same small distance from $U_2$, while the frequency $(1,k)$ separates the $k$th law from every other one by at least $q(1-q)$ (Figure~\ref{fig:shear-sequence}). Their likelihood ratios against uniformity have a common finite second moment. This yields a direct estimation lower bound by two applications of Cauchy--Schwarz. Here $\lambda_{\min}(\Sigma_k)\sim a/k^2$: Fourier mass of fixed size moves to arbitrarily high frequencies. In contrast, a fixed floor $\Sigma\succeq\varepsilon I_d$ imposes the common envelope $|m_n|\le e^{-\varepsilon|n|^2/2}$, making the closure compact and Fourier estimation uniform. Divergence of the condition number alone is insufficient to cause the obstruction: $\operatorname{diag}(\varepsilon,t)$ has unbounded condition number but lies in the regularized class.

\begin{figure}[htbp]
\centering
\includegraphics[width=\textwidth]{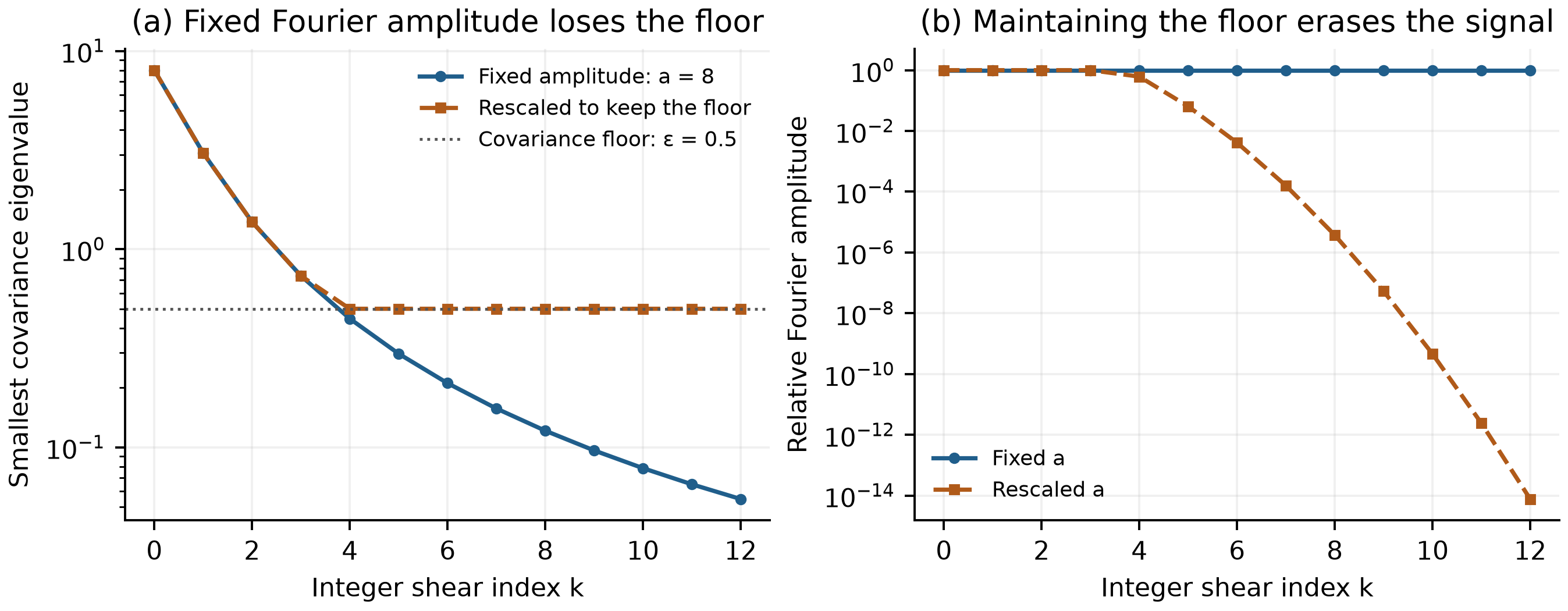}
\caption{The covariance floor and the shear mechanism. For $B_k=\left(\begin{smallmatrix}1+k^2&-k\\-k&1\end{smallmatrix}\right)$, compare $\Sigma_k=8B_k$ with $a_kB_k$, where $a_k=\max\{8,0.5/\lambda_{\min}(B_k)\}$. The matched frequency $(1,k)$ has moment magnitude $e^{-a_k/2}$. Panel (b) divides this magnitude by its initial value $e^{-4}$. Keeping the smallest covariance eigenvalue above $0.5$ forces this Fourier amplitude to vanish. The plotted values are computed from these matrices.}\label{fig:intro-covariance}
\end{figure}

We also determine some metric behavior at the added set. Every added law lies on an isometrically embedded small circular disk. Opposite radial points are joined by a minimizing geodesic through that law; when its residual rank is $d-1$, the two endpoints are finite-parameter laws and their distance is not attained by any path wholly in the finite family (Theorem~\ref{thm:addedpointgeodesics}). In a different configuration, two independent transverse factors admit a shorter product path, so their common uniform factor cannot be a minimizing transit point. Corollary~\ref{cor:crosspiecebridge} gives a sharp right-angle distance asymptotic for these compatible pieces. These results resolve specific instances of the added-point distance and geodesic questions without asserting existence for every pair in the nonproper completion.

\subsection*{Context and contribution}
Wrapped Gaussian models and their Fourier representations are standard in directional statistics \cite{MardiaJupp1999,JammalamadakaSenGupta2001,LeyVerdebout2017}; multivariate inference is developed in \cite{Coles1998,JonaLasinio2012,Nodehi2021,Saraceno2021}. Gaussian Fisher geometry supplies a complete comparison model \cite{AmariNagaoka2000,CostaSantosStrapasson2015,GlobkeQuiroga2021}. Information loss under wrapping is an instance of the classical missing-information principle \cite{OrchardWoodbury1972,Louis1982}. Work on spaces of densities \cite{BruverisMichor2019,Bauer2024} and closures of exponential families \cite{CsiszarMatus2005} provides related settings, but does not determine the intrinsic closure of this wrapped subfamily.

It is useful to separate what follows from the topology alone from what depends on the Fisher distance. Since the completion is homeomorphic to the total-variation closure, the description of the limit laws as a set, the compactness criterion of Proposition~\ref{prop:compactcriterion}, and the failure of local compactness as a topological property could be phrased for total variation. The Fisher geometry enters in the following places: every limit law is reached by paths of finite Fisher length inside the model, and different approaches to it define one metric endpoint; the distance to a limit has the exact leading coefficient $\sqrt{N(S)}$; the separation at a fixed relative scale, and hence the infinite Assouad dimension and the Hausdorff dimensions, refer to $d_{\mathrm W}$; the minimizing geodesics through added points and their nonattainment inside $\mathcal M_d$ are statements about $d_{\mathrm W}$; and the minimax obstruction is stated on balls of the completed Fisher distance.

The tools used here---Fourier series, lattice Gaussian sums, conditional scores, data processing, and elementary testing and projection estimates---are classical. The contribution is the explicit intrinsic classification, uniform control of approaches with unbounded cross-covariance, and the resulting link between local geometry and estimability. The positive regularization theorem uses the standard orthogonal-series bias--variance argument \cite{Tsybakov2009}; its role is to identify precisely what a covariance floor changes in this model. The packing argument also explains why finite Hausdorff dimension need not imply local doubling or manageable statistical complexity \cite{Heinonen2001,Fraser2020,LeCam1973,Yu1997}.

In one observation dimension the family is the classical circular wrapped normal family. Theorem~\ref{thm:completionclassification} with $d=1$ shows that its completion adds only the uniform law, and Lemma~\ref{lem:localcirculardisk} shows that the Fisher metric extends real analytically and positively across that point. These two facts are all that is used about the circular family; both are proved here.

Section~\ref{sec:fullcov} defines the model and the Gaussian comparison. Section~\ref{sec:classification} proves the classification and lattice-distance estimates. Section~\ref{sec:contractions} develops heat contractions, canonical retractions, and the metric results across pieces. Section~\ref{sec:topology} proves the local packing obstruction, and Section~\ref{sec:statistics} establishes its statistical consequence and the spectral-floor recovery theorem.
\section{General covariance and the Gaussian comparison}\label{sec:fullcov}
Fix an integer $d\ge1$ and the torus $\T^d=\R^d/(2\pi\Z^d)$ with its fixed coordinates. Let $\mathcal P_d$ denote the positive-definite symmetric $d\times d$ matrices and put
\[
 \mathcal M_d=\T^d\times\mathcal P_d.
\]
A point $([\mu],\Sigma)$ parametrizes both an orbit of real normal distributions and the law of $Y=X\bmod2\pi$, where $X\sim N_d(\mu,\Sigma)$. All covariance entries are free to vary. Denote the quotient metric by $g_{\mathrm Q}$, the wrapped Fisher metric by $g_{\mathrm W}$, and their intrinsic distances by $d_{\mathrm Q},d_{\mathrm W}$. Let $\widehat{\mathcal M}_{d,\mathrm W}$ be the metric completion of the latter. We identify means modulo the period lattice only; integer linear changes of torus coordinates act on these spaces but are not additional identifications.

\subsection{Finite parameters and strict information loss}
Write $\dd\nu=(2\pi)^{-d}\dd\theta$ for normalized Haar measure. The wrapped density relative to $\nu$ and its moments are
\begin{align}
 F_{\mu,\Sigma}(\theta)
 &=\sum_{n\in\Z^d}e^{-n^\top\Sigma n/2}e^{in^\top(\theta-\mu)},\label{eq:fullfourier}\\
 m_n(\mu,\Sigma)&=\E[e^{in^\top Y}]
 =e^{in^\top\mu-n^\top\Sigma n/2}.\label{eq:fullmoments}
\end{align}
The characteristic function of the real normal gives the moment formula, and Fourier inversion gives the density. Equivalently it is $(2\pi)^d$ times the sum of the real Gaussian density over all integer lifts; this proves positivity and normalization. The multivariate wrapping construction is standard; see, for example, \cite{JonaLasinio2012,Nodehi2021}.

For later product constructions, write the normalized one-dimensional density as
\[
 h_q(u)=\sum_{j\in\Z}q^{j^2}e^{iju}
       =1+2\sum_{j\ge1}q^{j^2}\cos(ju),\qquad 0\le q<1,
\]
with $h_0=1$. This is \eqref{eq:fullfourier} with $d=1$, mean zero, and variance $-2\log q$ when $q>0$.

\begin{proposition}[Regularity and information comparison]\label{prop:fullregular}
The wrapped parametrization is identifiable on $\mathcal M_d$ and $g_{\mathrm W}$ is positive definite and real analytic. For a tangent $V=(v,H)$, with $H$ symmetric,
\begin{equation}\label{eq:fullgaussian}
 g_{\mathrm Q}(V,V)
 =v^\top\Sigma^{-1}v
   +\frac12\operatorname{tr}(\Sigma^{-1}H\Sigma^{-1}H).
\end{equation}
For every $V\ne0$,
\begin{equation}\label{eq:fullstrict}
 0<g_{\mathrm W}(V,V)<g_{\mathrm Q}(V,V).
\end{equation}
The mean--covariance mixed block of the wrapped Fisher matrix vanishes, although neither block generally separates into independent coordinates.
\end{proposition}
\begin{proof}
The moments $m_{e_j}$ determine $\mu$ modulo $2\pi$ and the diagonal covariance entries. Writing $L_n=-2\log|m_n|$ gives
\[
 \Sigma_{ij}=\frac12(L_{e_i+e_j}-L_{e_i}-L_{e_j}),\qquad i\ne j,
\]
so all covariance entries are determined. Moreover,
\[
 \dd m_n[V]=m_n\left(in^\top v-\frac12n^\top Hn\right).
\]
A tangent of zero wrapped Fisher norm has zero density derivative and hence zero moment derivatives. Since $m_n\ne0$ at every finite parameter, divide the displayed identity by $m_n$. Taking imaginary parts for $n=e_j$ then gives $v=0$; taking real parts for $e_j$ and $e_i+e_j$ gives $H=0$.

On compact positive-definite parameter sets the Fourier series and all finite derivatives converge normally. Complexifying $\mu$ and the independent entries of $\Sigma$ in a small neighborhood preserves a positive lower bound for the real part of the quadratic form in $n$. Its quadratic decay dominates the linear growth from the imaginary mean. Positivity on the compact real torus implies that the complexified density stays nonzero in a smaller neighborhood. Integrating products of derivatives divided by this density proves real analyticity of the metric. After shifting by $\mu$, the density is even under $\theta\mapsto-\theta$. Mean derivatives are odd and covariance derivatives are even, proving the mixed-block assertion.

Set $Z=X-\mu$. The real Gaussian directional score is
\[
 U_V(X)=v^\top\Sigma^{-1}Z
 +\frac12\left(Z^\top\Sigma^{-1}H\Sigma^{-1}Z
                   -\operatorname{tr}(\Sigma^{-1}H)\right).
\]
Its variance is \eqref{eq:fullgaussian}, by the Gaussian second and fourth moments. The metric coefficients in \eqref{eq:fullgaussian} are independent of $\mu$, so the metric descends under period translations. Differentiating the lift sum, justified by locally uniform Gaussian domination, yields
\[
 U_{\mathrm W,V}(Y)=\E[U_V(X)\mid Y],\qquad
 g_{\mathrm Q}(V,V)-g_{\mathrm W}(V,V)
       =\E[\Var(U_V(X)\mid Y)].
\]
At each observed $y$, all lifts $y+2\pi k$ have positive conditional probability. Zero conditional variance would make the quadratic polynomial $U_V(y+2\pi k)$ constant for every $k\in\Z^d$. Restricting first to $k=te_i$, $t\in\Z$, annihilates the diagonal quadratic and linear coefficients. Restricting next to $k=t(e_i+e_j)$ annihilates the off-diagonal coefficients. Thus $\Sigma^{-1}H\Sigma^{-1}=0$ and subsequently $\Sigma^{-1}v=0$. This contradicts $V\ne0$ and proves strictness.
\end{proof}

\subsection{The complete Gaussian quotient}
\begin{proposition}[Completeness without hyperbolicity]\label{prop:fullcomplete}
The Gaussian parameter quotient $(\mathcal M_d,g_{\mathrm Q})$ is complete and its closed metric balls are compact. If a curve of length $L$ starts at covariance $\Sigma_0$, its final covariance satisfies
\begin{equation}\label{eq:loewnerlength}
 e^{-\sqrt2L}\Sigma_0\preceq\Sigma\preceq e^{\sqrt2L}\Sigma_0.
\end{equation}
For $d\ge2$ its metric is not of constant negative sectional curvature.
\end{proposition}
\begin{proof}
For every fixed $x\ne0$, along a piecewise smooth curve,
\begin{align*}
 \left|\frac{\mathrm d}{\mathrm dt}\log(x^\top\Sigma x)\right|
 &\le \|\Sigma^{-1/2}\dot\Sigma\Sigma^{-1/2}\|_{\mathrm{op}}\\
 &\le \|\Sigma^{-1/2}\dot\Sigma\Sigma^{-1/2}\|_{\mathrm F}
 \le\sqrt2\,\sqrt{g_{\mathrm Q}(\dot\gamma,\dot\gamma)}.
\end{align*}
Integration for all $x$ proves \eqref{eq:loewnerlength}. Approximate minimizing curves then give the same bound with $L$ replaced by distance. Covariances in a bounded metric ball therefore lie in a compact subset of $\mathcal P_d$, and means lie in the compact torus. Since the smooth positive metric induces the manifold topology, the closed ball is a closed subset of that compact parameter set and is compact. Every Cauchy sequence is bounded and hence converges.

For the last assertion, coordinate sign changes are isometries of \eqref{eq:fullgaussian}. The component of their common fixed set with $\mu=0$ consists of diagonal covariances. It is totally geodesic: each sign change fixes the initial data of a geodesic tangent to that component, so uniqueness of geodesics keeps it in the fixed set. For $\Sigma=\diag(\lambda_1,\ldots,\lambda_d)$ the induced metric is $\frac12\sum_j(\dd\log\lambda_j)^2$. This is flat and has two-dimensional tangent planes when $d\ge2$. Total geodesy makes their ambient sectional curvature zero, ruling out constant negative curvature. The broader geometry of the Gaussian Fisher metric is discussed in \cite{GlobkeQuiroga2021}.
\end{proof}

In one observation dimension the contrast between the two constructions can be seen directly. The quotient metric is $s^{-2}(\dd\alpha^2+2\dd s^2)$, so the coordinate circles shrink as $s\to\infty$ but recede to infinite distance. In the wrapped family the uniform law lies at finite distance (Theorem~\ref{thm:completionclassification} with $d=1$), and the metric is smooth and positive there (Lemma~\ref{lem:localcirculardisk}). The coordinate circles therefore close at a finite distance, and since $g_0=2(\dd x^2+\dd y^2)$ their length is asymptotic to $2\pi$ times the remaining distance; see Figure~\ref{fig:cap}.

\begin{figure}[ht]
\centering
\begin{tikzpicture}
\begin{axis}[width=.92\textwidth,height=6.2cm,xlabel={Intrinsic distance toward larger standard deviation},ylabel={Coordinate-circle length},xmin=0,xmax=3.5,ymin=0,ymax=5.5,grid=major,grid style={gray!15},legend style={font=\footnotesize,at={(.98,.97)},anchor=north east}]
\addplot[blue!55!black,very thick] coordinates {(0.0,4.725738577373219) (0.015589324095328316,4.636254168090057) (0.03102426731112584,4.54724753333375) (0.04630951129164884,4.458705259853601) (0.06144952787908886,4.370615850107153) (0.07644859434557327,4.28296947230111) (0.09131080751990617,4.195757734200854) (0.1060400968929831,4.108973478715422) (0.12064023678807134,4.0226105993984085) (0.13511485767430675,3.9366638741351943) (0.14946745669468853,3.851128815411928) (0.1637014074734371,3.766001535681825) (0.1778199692617638,3.6812786264588477) (0.1918262954757911,3.59695704987765) (0.20572344167554057,3.5130340415616255) (0.2195143730294662,3.429507023737956) (0.2332019713049741,3.3463735276297815) (0.24678904142164337,3.263631124241024) (0.26027831760044395,3.1812773627292903) (0.27367246913910304,3.0993097156366622) (0.2869741058408701,3.0177255303173585) (0.30018578312125094,2.9365219859654568) (0.3133100068148196,2.855696055705237) (0.32634923770193097,2.7752444732615555) (0.33930589577306575,2.695163703778193) (0.35218236424659677,2.6154499183985185) (0.36498099335399814,2.5360989722653953) (0.37770410390487813,2.4571063856361133) (0.39035399064273907,2.3784673278435555) (0.4029329254010089,2.3001766038669547) (0.41544316006766807,2.2222286433046374) (0.42788692936570133,2.1446174915673213) (0.44026645345562687,2.0673368031338706) (0.4525839403655061,1.9903798367322558) (0.46484158825309285,1.9137394523267897) (0.4770415875041622,1.837408109808764) (0.48918612267054007,1.76137786930152) (0.5012773742509489,1.68564039300286) (0.5133175203174813,1.6101869484976676) (0.5253087379903113,1.5350084134818676) (0.5372532047631412,1.460095281845401) (0.5491530996818753,1.3854376710670109) (0.5610106043790745,1.311025330877301) (0.5728279039669022,1.2368476531489878) (0.584607187791502,1.1628936829745582) (0.5963506500520378,1.0891521308918166) (0.608060490287986,1.0156113862172227) (0.6197389137386822,0.9422595314454951) (0.6313881315795677,0.869084357671953) (0.6430103610400872,0.7960733809914419) (0.6546078254086833,0.723213859824714) (0.6661827539308852,0.6504928131197822) (0.6777373816070168,0.5778970393722308) (0.689273948896586,0.5054131364048402) (0.7007947013369495,0.43302752184320575) (0.7123018890843412,0.3607264542204854) (0.7237977663858308,0.288496054640987) (0.7352845909912104,0.2163223289291741) (0.7467646235141959,0.14419119018781393) (0.7582401267526627,0.07208848168652934) (0.7697133649779095,0.0)};
\addlegendentry{Wrapped model}
\addplot[red!70!black,dashed,thick,domain=0:3.5,samples=100] {2*pi/1.2*exp(-x/sqrt(2))};
\addlegendentry{Gaussian parameter quotient}
\addplot[orange!85!black,dotted,thick,domain=.38:.769713364978] {2*pi*(.769713364978-x)};
\addlegendentry{Tangent at $\delta=D_*$, slope $-2\pi$}
\end{axis}
\end{tikzpicture}
\caption{One observation dimension: lengths of the coordinate circles $\{s\}\times\T$ for $s\ge s_0=1.2$, plotted against the intrinsic distance from the circle at $s_0$ in each model. The wrapped circle closes at the uniform law, at distance $D_*\approx0.7697$, with limiting slope $-2\pi$; the quotient circle closes only at infinite distance. The wrapped curve is computed numerically from the Fisher coefficients and is an illustration only.}\label{fig:cap}
\end{figure}
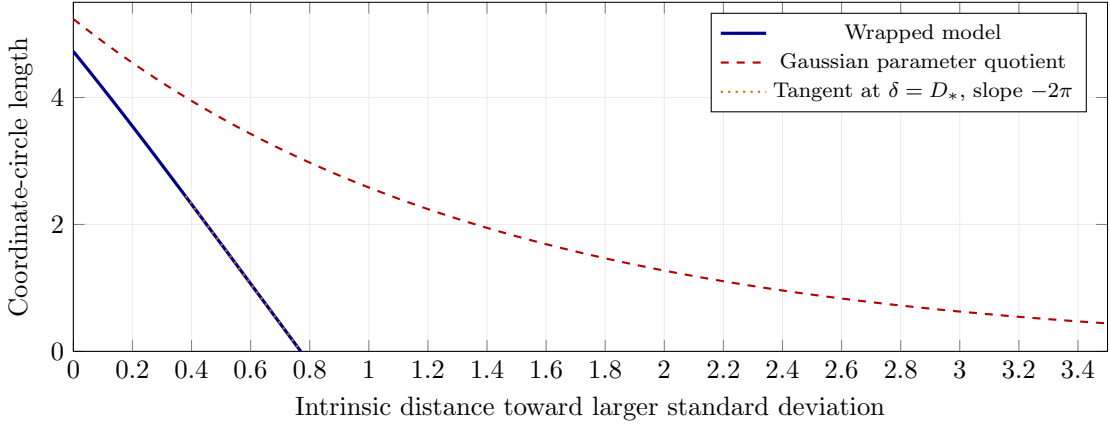

\section{The full intrinsic Fisher completion}\label{sec:classification}
Write $\widehat{\mathcal M}_{d,\mathrm W}$ for the metric completion and use $d_{\mathrm W}$ also for its completed distance. All density norms below use normalized Haar measure $\nu$. Our convention for total variation is $d_{\mathrm{TV}}(F,G)=\frac12\|F-G\|_1$.

For the exact parametrization of the limiting laws, call a subgroup $L\le\Z^d$ \emph{primitive} if $kn\in L$, for $n\in\Z^d$ and a nonzero integer $k$, implies $n\in L$. For such an $L$, a character $\chi:L\to S^1$, and a positive-definite quadratic form $Q$ on $\operatorname{span}_{\R}L$, define $F_{L,\chi,Q}$ by its moments
\begin{equation}\label{eq:classifiedmoments}
 m_n(F_{L,\chi,Q})=
 \begin{cases}\chi(n)e^{-Q(n)/2},&n\in L,\\0,&n\notin L.\end{cases}
\end{equation}
For $L=\{0\}$ this is the uniform density. An integer change of torus coordinates expresses this law as $U_{d-r}\otimes\mathrm{WN}_r(\beta,B)$, where $r=\operatorname{rank}L$ and $B$ is positive definite. Let $\mathcal C_d$ be the set of these densities, with its $L^1$ topology, equivalently its total-variation topology. Integer coordinate changes describe the laws; they do not impose additional identifications on the fixed torus.

\begin{example}[The case $d=2$]\label{ex:d2}
For $X\sim N_2(\mu,\Sigma)$ and $n\in\Z^2$, the angle $n\cdot Y$ is the wrapped image of the real normal variable $n\cdot X$, with mean $n\cdot\mu$ and variance $n^\top\Sigma n$. Thus $|m_n|=e^{-n^\top\Sigma n/2}$ measures the concentration of this angle alone. Along a sequence of parameters, each frequency either survives, with $n^\top\Sigma_jn$ bounded, or disappears, with $n^\top\Sigma_jn\to\infty$. The proof of Lemma~\ref{lem:densityclosure} shows that, for Gaussian limits, the surviving frequencies form a primitive subgroup: sums and negatives of surviving frequencies survive, and $Q(kn)=k^2Q(n)$ rules out a surviving multiple $kn$ of a disappearing $n$. In $\Z^2$ the only possibilities are $L=\{0\}$, $L=\Z n$ for a primitive $n=(p,q)$ taken up to sign, and $L=\Z^2$. The last gives the finite family. The added points are therefore the uniform law $U_2$ and, for each primitive direction $n$, the laws under which $pY_1+qY_2$ is a circular wrapped normal $\mathrm{WN}_1(\beta,b)$ while the complementary integer coordinate is uniform and independent; here $\chi(kn)=e^{ik\beta}$ and $Q(kn)=k^2b$.

Figure~\ref{fig:latticeexample} shows four parameter sequences, with $Z\sim N(0,t)$ and $W\sim N(0,b)$ independent. For $\Sigma=tI_2$ every nonzero frequency disappears, and the limit is $U_2$. For $\Sigma=\operatorname{diag}(t,b)$ the frequencies $(0,q)$ survive, so $L=\Z(0,1)$ and $Y_2$ remains a wrapped normal. For $X=(Z,Z+W)$ one has $n\cdot X=(n_1+n_2)Z+n_2W$, so exactly the frequencies with $n_1+n_2=0$ survive; then $L=\Z(1,-1)$ and $Y_2-Y_1$ remains a wrapped normal, although neither coordinate does. The last panel uses $\Sigma_j=jB_j$ with $B_j=\left(\begin{smallmatrix}1+j^2&-j\\-j&1\end{smallmatrix}\right)$. Since $B_j$ is unimodular with lattice minimum $1$, every nonzero frequency satisfies $n^\top\Sigma_jn\ge j$, so the limit is again $U_2$, although $\lambda_{\min}(\Sigma_j)\sim1/j\to0$. Convergence is therefore governed by the quadratic form on the integer lattice, not by the Euclidean eigenvalues; this is the content of the criterion in Theorem~\ref{thm:completionclassification}.

For each primitive $n$, the laws with $L=\Z n$ form a two-parameter family indexed by $(\beta,b)$; Proposition~\ref{prop:pieces} below identifies its closure with the completed circular family, a disk centered at $U_2$. Distinct disks meet only at $U_2$. The completion also contains the five-dimensional finite family, and distances between different disks may use it, so the added set does not carry the distance of a wedge of disks.
\end{example}

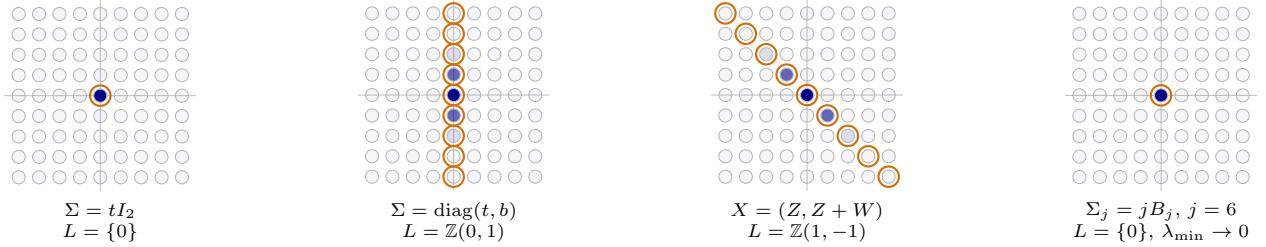
\begin{figure}[ht]
\centering
\newcommand{\latticepanel}[5]{%
\begin{tikzpicture}[scale=0.27]
\draw[gray!50,line width=0.3pt] (-4.7,0)--(4.7,0); \draw[gray!50,line width=0.3pt] (0,-4.7)--(0,4.7);
\foreach \p in {-4,...,4}{\foreach \q in {-4,...,4}{%
\pgfmathsetmacro{\v}{\p*\p*(#1)+2*\p*\q*(#2)+\q*\q*(#3)}%
\pgfmathsetmacro{\o}{max(0.04,exp(-\v/2))}%
\fill[blue!55!black,fill opacity=\o] (\p,\q) circle (0.32);
\draw[gray!60,line width=0.2pt] (\p,\q) circle (0.32);}}
#4
\node[font=\scriptsize,align=center] at (0,-6.2) {#5};
\end{tikzpicture}}
\latticepanel{6}{0}{6}{\draw[orange!85!black,line width=0.8pt] (0,0) circle (0.5);}{$\Sigma=tI_2$\\ $L=\{0\}$}\hfill
\latticepanel{6}{0}{1}{\foreach \q in {-4,...,4}{\draw[orange!85!black,line width=0.8pt] (0,\q) circle (0.5);}}{$\Sigma=\operatorname{diag}(t,b)$\\ $L=\Z(0,1)$}\hfill
\latticepanel{6}{6}{7}{\foreach \q in {-4,...,4}{\draw[orange!85!black,line width=0.8pt] (-\q,\q) circle (0.5);}}{$X=(Z,Z+W)$\\ $L=\Z(1,-1)$}\hfill
\latticepanel{222}{-36}{6}{\draw[orange!85!black,line width=0.8pt] (0,0) circle (0.5);}{$\Sigma_j=jB_j$, $j=6$\\ $L=\{0\}$, $\lambda_{\min}\to0$}
\caption{Surviving frequencies for $d=2$. Each disk is an integer frequency $n=(n_1,n_2)$, shaded by $|m_n|=e^{-n^\top\Sigma n/2}$ at $t=6$, $b=1$ (last panel: $j=6$). Rings mark the frequencies that survive in the limit $t\to\infty$ or $j\to\infty$; the shading is a single snapshot. The surviving set is the subgroup $L$ of Theorem~\ref{thm:maincompletion}. In the last panel the smallest covariance eigenvalue tends to zero, yet every nonzero frequency disappears.}\label{fig:latticeexample}
\end{figure}

\subsection{The statistical closure}
\begin{lemma}[Lipschitz observables]\label{lem:observables}
For a bounded complex observable $\varphi$ on the torus, $P\mapsto\E_P\varphi$ is $\|\varphi\|_\infty$-Lipschitz for $d_{\mathrm W}$. In particular,
\begin{equation}\label{eq:momentlower}
 |m_n(P)-m_n(P')|\le\|F_P-F_{P'}\|_1\le d_{\mathrm W}(P,P'),\qquad n\in\Z^d.
\end{equation}
The density map extends continuously to the completion, with values in the nonnegative $L^1$ densities of integral one.
\end{lemma}
\begin{proof}
If $U_V$ is a directional score, Cauchy--Schwarz gives
$|\dd\E_P\varphi[V]|\le\|\varphi\|_\infty\sqrt{g_{\mathrm W}(V,V)}$ and
$\int|\dd F[V]|\dd\nu=\E_P|U_V|\le\sqrt{g_{\mathrm W}(V,V)}$.
Integrating along connecting curves and taking infima proves the distance bounds. The moment bound by $L^1$ follows directly by integration. Completeness of $L^1$ gives the extension, and an $L^1$ limit of probability densities is again a probability density.
\end{proof}

\begin{lemma}[Closure of the densities]\label{lem:densityclosure}
The $L^1$ closure of $\{F_{\mu,\Sigma}:(\mu,\Sigma)\in\mathcal M_d\}$ is precisely $\mathcal C_d$ from \eqref{eq:classifiedmoments}. Every density in this closure is strictly positive and real analytic, and its triple $(L,\chi,Q)$ is unique.
\end{lemma}
\begin{proof}
Suppose $F_j\to F$ in $L^1$, and put $Q_j(n)=n^\top\Sigma_j n$ and
$L=\{n\in\Z^d:|m_n(F)|>0\}$. Moment convergence implies
\[
 Q_j(n)\longrightarrow-2\log|m_n(F)|<\infty\quad(n\in L),
 \qquad Q_j(n)\longrightarrow\infty\quad(n\notin L).
\]
The inequality $Q_j(n+p)\le2Q_j(n)+2Q_j(p)$ shows that $L$ is a subgroup. The identity $Q_j(kn)=k^2Q_j(n)$ shows that it is primitive. A basis of a primitive subgroup extends to an integer basis of $\Z^d$. Choose $T\in\mathrm{GL}(d,\Z)$ whose last $r$ rows are a basis of $L$, where $r=\operatorname{rank}L$, and set $m=d-r$. Under the observation transformation $\theta\mapsto T\theta$, write
\begin{equation}\label{eq:adaptedblocks}
 \mu_j=(\alpha_j,\beta_j),\qquad
 \Sigma_j=\begin{pmatrix}A_j&C_j\\C_j^\top&B_j\end{pmatrix},
 \qquad S_j=A_j-C_jB_j^{-1}C_j^\top.
\end{equation}
Thus $L$ is the last $r$-coordinate frequency subgroup. The surviving moments at the residual frequencies $e_{m+i}$ and $e_{m+i}+e_{m+l}$ recover $B_j\to B\succeq0$ by polarization, and their phases recover $\beta_j\to\beta$ in $\T^r$. The limiting phases on $L$ form a character. Fourier uniqueness identifies the limiting law as $U_m\otimes\mathrm{WN}_r(\beta,B)$, initially allowing a degenerate Gaussian.

In fact $B>0$. If its rank is less than $r$, its Euclidean Gaussian is supported on a lower-dimensional affine subspace. The image of this subspace modulo $2\pi\Z^r$ is a countable union of Lipschitz images of bounded lower-dimensional sets, hence has Haar measure zero, even if it is dense. The wrapped law is singular. This contradicts absolute continuity of the residual marginal of $F$. If $r=0$, the limit is simply uniform. If $r=d$, the same argument shows that the limit belongs to the original family.

Conversely, in these coordinates each stated law is the uniform limit of product densities with mean $(0,\beta)$ and covariance $\diag(aI_m,B)$ as $a\to\infty$. The product representation also proves positivity and real analyticity. Finally $L$ is recovered from the nonzero moments, $Q(n)=-2\log|m_n|$ on $L$, and $\chi(n)=m_n/|m_n|$ there. Values on a lattice basis and its pairwise sums determine the quadratic form on its real span, proving uniqueness.
\end{proof}

This lemma identifies distributions, not yet metric endpoints. We next show that all these limits are reached by finite-length curves and that approaches with the same limiting density agree in the metric completion.

\subsection{Uniform intrinsic distance estimates}
For a positive-definite $m\times m$ matrix $S$, $m\ge1$, define
\begin{equation}\label{eq:latticemin}
 \ell(S)=\min_{k\in\Z^m\setminus\{0\}}k^\top Sk,
 \qquad \Theta_m(S)=\sum_{k\ne0}e^{-k^\top Sk/2}.
\end{equation}
\begin{lemma}[A lattice Gaussian bound]\label{lem:latticetheta}
There is $c_m<\infty$ depending only on $m$ such that
$\Theta_m(S)\le c_me^{-\ell(S)/2}$ whenever $\ell(S)\ge1$.
\end{lemma}
\begin{proof}
The lattice $S^{1/2}\Z^m$ has minimum separation $\rho=\sqrt{\ell(S)}$. Disjoint balls of radius $\rho/2$ centered at its points show that at most $(2h+3)^m$ points have norm less than $(h+1)\rho$. Counting shells $h\rho\le|x|<(h+1)\rho$, $h\ge1$, gives
\[
 \Theta_m(S)\le\sum_{h\ge1}(2h+3)^m e^{-h^2\ell(S)/2}
 \le e^{-\ell(S)/2}\sum_{h\ge1}(2h+3)^m e^{-(h^2-1)/2}.
\]
The final series is finite.
\end{proof}

Gaussian sums controlling approximation to uniformity modulo a lattice also appear in the estimates of Micciancio and Regev for the smoothing parameter $\eta_\epsilon$ of a lattice \cite{MicciancioRegev2007}. Here the additional requirement is a bound for lengths of paths inside the fixed wrapped family. In particular, the estimate must allow arbitrary cross-covariance and cannot assume a uniform bound on the condition number.

Fix $0\le r<d$, put $m=d-r$, and work in fixed adapted integer coordinates. For $r>0$ let $B\in\mathcal P_r$, $\beta\in\T^r$, and denote the residual wrapped density by $f_{\beta,B}$. Set
\[
 c_B=\min f_{\beta,B}>0,\qquad
 K_B=\sup_{x\in\R^r}\sum_{l\in\Z^r}e^{-(l+x)^\top B(l+x)/2}.
\]
Both are independent of $\beta$. The quantities $K_B$ and $c_B^{-1}$ are bounded on compact subsets of $\mathcal P_r$. Indeed, the shifted sum is periodic in $x$ and has uniformly summable Gaussian tails there; positivity and compactness give the density bound. The same argument applies to sums with any fixed polynomial weight in $|l+x|$. For $r=0$ omit residual variables and set $c_B=K_B=1$.

\begin{proposition}[Distance to a partial uniform law]\label{prop:endpointdistance}
The product curve with mean $(0,\beta)$ and covariance $\diag(tI_m,B)$, $t\ge1$, has a completion endpoint $E(\beta,B)$ with density $f_{\beta,B}$, constant in the first $m$ variables. For
\begin{equation}\label{eq:schurparam}
 P=\left((\alpha,\beta),
 \begin{pmatrix}S+D^\top BD&D^\top B\\BD&B\end{pmatrix}\right),
 \qquad S\in\mathcal P_m,\quad D\in\R^{r\times m},
\end{equation}
the curve obtained by replacing $S$ by $tS$, $t\ge1$, has finite length and tends to this same endpoint. Whenever $\ell(S)$ is sufficiently large,
\begin{equation}\label{eq:partialdistance}
 e^{-r\|B\|_{\mathrm{op}}/8}e^{-\ell(S)/2}
 \le d_{\mathrm W}(P,E(\beta,B))
 \le C_{B,m}e^{-\ell(S)/2}.
\end{equation}
The threshold and upper constant can be chosen uniformly for $B$ in a compact subset of $\mathcal P_r$, independently of $D,\alpha,\beta$. For $r=0$ the left prefactor is one.
\end{proposition}
\begin{proof}
At frequency $(k,l)\in\Z^m\times\Z^r$, the covariance quadratic form along the curve is
$t k^\top Sk+(l+Dk)^\top B(l+Dk)$. If $F_t$ is its density, summation first over $l$ gives
\begin{align*}
 \|F_t-f_{\beta,B}\|_\infty&\le K_B\Theta_m(tS),\\
 \|\dot F_t\|_\infty&\le\frac{K_B}{2}
       \sum_{k\ne0}(k^\top Sk)e^{-t k^\top Sk/2}.
\end{align*}
If $K_B\Theta_m(S)\le c_B/2$, then $F_t\ge c_B/2$ for $t\ge1$, so its Fisher speed is at most $\sqrt{2/c_B}\|\dot F_t\|_2$, and hence at most $\sqrt{2/c_B}\|\dot F_t\|_\infty$. Tonelli's theorem yields
\begin{equation}\label{eq:tailpathlength}
 \operatorname{Length}(F_t:1\le t<\infty)
 \le\sqrt{2/c_B}K_B\Theta_m(S).
\end{equation}
For arbitrary $S>0$, the same bound applies after a finite initial segment. Thus every such curve is Cauchy at infinity. Taking $S=I_m$, $D=0$, and $\alpha=0$ constructs $E(\beta,B)$.

We identify the other endpoints metrically before using their densities. Fix $S,D,\alpha$ and a real lift of $\alpha$. At each large $t$, replace $D$ and $\alpha$ by $D_u=(1-u)D$ and $\alpha_u=(1-u)\alpha$, $0\le u\le1$, leaving $tS,B,\beta$ fixed in \eqref{eq:schurparam}. The matrices remain positive definite. The same Fourier bound gives density at least $c_B/2$ uniformly in $u$. With
\[
 K_B^{(1)}=\sup_x\sum_l|l+x|e^{-(l+x)^\top B(l+x)/2}<\infty,
\]
the density derivative along the connector is bounded by
\[
 \bigl(|\alpha|K_B+\|BD\|_{\mathrm{op}}K_B^{(1)}\bigr)
 \sum_{k\ne0}|k|e^{-t k^\top Sk/2}.
\]
For fixed $S>0$ this tends to zero by Gaussian domination. Therefore the connector lengths tend to zero. Their final points have covariance $\diag(tS,B)$ and mean $(0,\beta)$.

Next connect these points to $\diag(tI_m,B)$ by replacing $S$ with $S_u=(1-u)S+uI_m$. Put $c=\min(\lambda_{\min}(S),1)>0$. These densities also have a common positive lower bound for large $t$, and their derivative is bounded by
\[
 \frac t2\|I_m-S\|_{\mathrm{op}}K_B
       \sum_{k\ne0}|k|^2e^{-tc|k|^2/2}\longrightarrow0.
\]
For example, absorbing $t$ into half the exponential gives a fixed summable Gaussian majorant. The connector lengths again tend to zero. Thus the original curve and the defining curve of $E(\beta,B)$ have the same completion endpoint. These connector arguments concern fixed $S,D,\alpha$; uniformity in them is not required for endpoint identification. The bound \eqref{eq:tailpathlength}, now a distance bound to that endpoint, is independent of $D$ and $\alpha$. Lemma~\ref{lem:latticetheta} proves the upper estimate in \eqref{eq:partialdistance}.

For the lower estimate, choose an integer $k\ne0$ minimizing $k^\top Sk$ and round $-Dk$ coordinatewise to an integer $l$. Then
$(l+Dk)^\top B(l+Dk)\le r\|B\|_{\mathrm{op}}/4$.
The endpoint has zero $(k,l)$ moment, whereas the magnitude of the moment at $P$ is at least $e^{-\ell(S)/2-r\|B\|_{\mathrm{op}}/8}$. Apply \eqref{eq:momentlower}. The uniformity in $B$ follows from the compact bounds for $K_B,c_B^{-1}$.
\end{proof}

\begin{proposition}[Sharp asymptotics on fixed conditional-covariance rays]\label{prop:sharpendpoint}
Fix $S\in\mathcal P_m$, $B\in\mathcal P_r$, $D\in\R^{r\times m}$,
and means $(\alpha,\beta)$ in \eqref{eq:schurparam}; for $r=0$ omit
the residual variables. Let $P_t$ be obtained by replacing $S$ with $tS$,
and let $E=E(\beta,B)$. Put
\[
 \ell=\ell(S),\qquad
 \mathcal Z=\{k\in\Z^m:k^\top Sk=\ell\},\qquad
 \ell_2=\ell_2(S)=\min\{k^\top Sk:k\in\Z^m,\ k^\top Sk>\ell\}.
\]
Both minima exist and $0<\ell<\ell_2$. With normalized Haar measure
on $\T^m\times\T^r$, define the real analytic function
\begin{equation}\label{eq:shortestshell}
 \Psi(x,y)=\sum_{k\in\mathcal Z}\sum_{l\in\Z^r}
 e^{-(l+Dk)^\top B(l+Dk)/2}
 e^{i\{k^\top(\alpha-x)+l^\top(\beta-y)\}},
 \qquad
 \mathcal A=\left(\int\frac{\Psi^2}{f_{\beta,B}}\dd\nu\right)^{1/2}>0.
\end{equation}
For $r=0$ the inner sum consists of its single empty-index term and
$f_{\beta,B}=1$. Set $\varrho=\frac12\min\{\ell,\ell_2-\ell\}$.
Then, as $t\to\infty$,
\begin{equation}\label{eq:sharpdistance}
 d_{\mathrm W}(P_t,E)
 =\mathcal A e^{-t\ell/2}\bigl(1+O(e^{-\varrho t})\bigr).
\end{equation}
The constants in this statement concern the fixed ray. In particular,
for covariance $tS$ in dimension $d$ and any fixed mean,
\begin{equation}\label{eq:kissingdistance}
 d_{\mathrm W}((\mu,tS),U_d)
 =\sqrt{N(S)}e^{-t\ell(S)/2}
    \bigl(1+O(e^{-\varrho t})\bigr),\qquad
 N(S)=\#\{n\in\Z^d:n^\top Sn=\ell(S)\}.
\end{equation}
\end{proposition}
\begin{proof}
Pairing $(k,l)$ with $(-k,-l)$ makes $\Psi$ real. Its series is normally
convergent, since $\mathcal Z$ is finite and the remaining sum has Gaussian
decay. A Fourier coefficient with $k\in\mathcal Z$ is nonzero, so $\Psi$ is
not identically zero and $\mathcal A>0$.

Writing $f=f_{\beta,B}$ and $a_t=e^{-t\ell/2}$, split the Fourier
series of $F_t$ into $k=0$, $k\in\mathcal Z$, and all remaining $k$.
The Gaussian bound used in Proposition~\ref{prop:endpointdistance} gives,
in the uniform norm,
\begin{equation}\label{eq:shellexpansions}
 F_t=f+a_t\Psi+O(e^{-t\ell_2/2}),\qquad
 \dot F_t=-\frac\ell2 a_t\Psi+O(e^{-t\ell_2/2}).
\end{equation}
For completeness, after factoring out $e^{-t\ell_2/2}$, the remainders
for $t\ge1$ are bounded by constant multiples of
\[
 \sum_{k:k^\top Sk\ge\ell_2}
 (1+k^\top Sk)e^{-(k^\top Sk-\ell_2)/2}<\infty.
\]
Summation over $l$ is uniformly bounded by $K_B$, including when the
shift $Dk$ varies with $k$. This also justifies the differentiated expansion.

Because $\min f>0$, equation~\eqref{eq:shellexpansions} and the smoothness
of the square-root and reciprocal functions on a positive interval imply
\[
 2\|\sqrt{F_t}-\sqrt f\|_2
 =\mathcal A a_t\bigl(1+O(e^{-\varrho t})\bigr),\qquad
 \left(\int\frac{\dot F_t^2}{F_t}\dd\nu\right)^{1/2}
 =\frac\ell2 \mathcal A a_t\bigl(1+O(e^{-\varrho t})\bigr).
\]
The relative errors come from the quadratic term of the square root,
of size $O(a_t)=O(e^{-t\ell/2})$, and from the next shell, of size
$O(e^{-t(\ell_2-\ell)/2})$.
The map $F\mapsto2\sqrt F$ takes every Fisher curve to an $L^2$ curve
of the same speed. Thus its chord length is a lower bound for intrinsic
Fisher distance. This bound extends to completion points: the inequality
$\|\sqrt F-\sqrt G\|_2^2\le\|F-G\|_1$ makes the square-root map
continuous there. The first displayed expansion is therefore a lower
bound for $d_{\mathrm W}(P_t,E)$. Integrating the second expansion from
$t$ to infinity gives the matching upper bound, since
Proposition~\ref{prop:endpointdistance} identifies the endpoint of this
ray with $E$. In the case $r=0$, orthogonality of distinct Fourier
characters gives $\int\Psi^2\dd\nu=\#\mathcal Z=N(S)$.
\end{proof}

The square-root representation is the usual ambient Fisher--Rao
representation; see \cite{BruverisMichor2019}. Here its chord only supplies
a lower bound, while the matching upper bound is produced inside the
wrapped family. Consequently \eqref{eq:sharpdistance} determines the
intrinsic leading constant without asserting that an ambient geodesic
lies in the family. For $S=I_d$, the constant in
\eqref{eq:kissingdistance} is $\sqrt{2d}$; for
$S=\left(\begin{smallmatrix}1&1/2\\1/2&1\end{smallmatrix}\right)$ it is
$\sqrt6$.

For each primitive $L$, choose one adapted integer basis to construct these endpoints. This is legitimate because a fixed integer change of observations preserves Fisher lengths in both directions. Independence of the chosen basis will follow from injectivity of the completed density map, rather than from density convergence alone.

\subsection{Classification and convergence}
\begin{theorem}[Full completion and convergence criterion]\label{thm:completionclassification}
The density map is a homeomorphism from $\widehat{\mathcal M}_{d,\mathrm W}$ onto $\mathcal C_d$. A finite-parameter sequence converges to the point $(L,\chi,Q)$ if and only if, in any fixed adapted coordinates \eqref{eq:adaptedblocks},
\begin{equation}\label{eq:completioncriterion}
 \beta_j\to\beta\text{ in }\T^r,\qquad B_j\to B>0,
 \qquad \ell(S_j)\to\infty.
\end{equation}
For $r=0$ only the last condition remains; for $r=d$ it is omitted. Equivalently, the characters and quadratic forms restricted to $L$ converge to $\chi,Q$, and
\begin{equation}\label{eq:outsidecriterion}
 \min_{n\in\Z^d\setminus L}n^\top\Sigma_j n\longrightarrow\infty
\end{equation}
when $L\ne\Z^d$. The character and form convergence can be checked on a basis and its pairwise sums. There is no additional restriction on $\alpha_j$ or $C_j$.
\end{theorem}
\begin{proof}
Suppose first that $F_{P_j}\to F\in\mathcal C_d$ in $L^1$, and fix adapted coordinates for $F$. Lemma~\ref{lem:densityclosure} gives $B_j\to B>0$ and $\beta_j\to\beta$. Put $\delta_j=\|F_{P_j}-F\|_1$. Every frequency $(k,l)$ with $k\ne0$ has zero limiting moment, so, uniformly over these frequencies,
\[
 \exp\!\left[-\tfrac12\{k^\top S_jk+(l+D_jk)^\top B_j(l+D_jk)\}\right]
 \le\delta_j,\qquad D_j=B_j^{-1}C_j^\top.
\]
Rounding $-D_jk$ to an integer $l$ shows
\begin{equation}\label{eq:roundingcriterion}
 \ell(S_j)\ge-2\log\delta_j-\frac r4\|B_j\|_{\mathrm{op}}\longrightarrow\infty.
\end{equation}
This uses the uniform moment bound supplied by $L^1$ convergence, not just convergence at each fixed frequency. Proposition~\ref{prop:endpointdistance} gives
$d_{\mathrm W}(P_j,E(\beta_j,B_j))\to0$.

For residual parameters, paths in $\mathcal M_r$ can be lifted by adjoining an independent wrapped factor of covariance $tI_m$. Independence makes the lifted speed equal to the residual speed. Passing to endpoints gives
\begin{equation}\label{eq:residuallift}
 d_{\mathrm W}(E(\beta_j,B_j),E(\beta,B))
 \le d_{\mathrm W,r}((\beta_j,B_j),(\beta,B))\longrightarrow0.
\end{equation}
Here $d_{\mathrm W,r}$ is the intrinsic wrapped distance in dimension $r$; its local convergence is ordinary parameter convergence by the smooth positive metric. For $r=0$ there is no residual term. For $r=d$, the finite moment inverse gives ordinary parameter convergence directly. Thus every $L^1$-convergent sequence of finite parameters converges intrinsically to the specified canonical point.

Every Fisher-Cauchy sequence has an $L^1$ density limit by Lemma~\ref{lem:observables}. The preceding argument identifies its metric endpoint uniquely. Conversely, Lemma~\ref{lem:densityclosure} and Proposition~\ref{prop:endpointdistance} construct an endpoint for every element of $\mathcal C_d$. The extended density map is therefore bijective. In particular, constructions in different adapted bases agree. Its continuity follows from \eqref{eq:momentlower}. For inverse continuity, suppose densities of completion points $x_j$ converge in $L^1$ to $F$. Choose finite $P_j$ with $d_{\mathrm W}(P_j,x_j)<1/j$. Their densities also tend to $F$, so the preceding argument and the triangle inequality give $x_j\to x_F$. Sequential continuity suffices since both spaces are metrizable.

The conditions in \eqref{eq:completioncriterion} imply intrinsic convergence by the same tail and residual estimates. Their necessity follows from intrinsic convergence, \eqref{eq:momentlower}, and \eqref{eq:roundingcriterion}. Finally, if $a_j$ denotes the minimum in \eqref{eq:outsidecriterion}, completing the square and rounding give
\[
 \ell(S_j)\le a_j\le\ell(S_j)+\frac r4\|B_j\|_{\mathrm{op}}.
\]
Since $B_j$ converge, the two divergence conditions are equivalent. Polarization and the nonzero surviving moments give the stated intrinsic description of the other conditions.
\end{proof}

\begin{corollary}[Local distance comparison with a fixed limit]\label{cor:anchoredcomparison}
For each completion point $E$, there are $\delta_E>0$ and $C_E<\infty$ such that
\begin{equation}\label{eq:anchoredcomparison}
 \|F_P-F_E\|_1\le d_{\mathrm W}(P,E)
 \le C_E\|F_P-F_E\|_1
 \quad\text{if }\|F_P-F_E\|_1<\delta_E.
\end{equation}
The point $P$ may belong to the completion. The constants concern the fixed endpoint $E$.
\end{corollary}
\begin{proof}
First let $P$ be finite. In adapted coordinates for $E$, its finitely many residual moments stay nonzero when the $L^1$ difference $\delta$ is small. The local inverse using their phases and logarithmic magnitudes bounds the residual parameter difference, and hence its local Fisher distance, by a constant times $\delta$. The residual covariances remain in a fixed compact subset of $\mathcal P_r$. Equation~\eqref{eq:roundingcriterion} gives $e^{-\ell(S_P)/2}\le e^{r\|B_P\|_{\mathrm{op}}/8}\delta$. Combine \eqref{eq:partialdistance} and \eqref{eq:residuallift}. For $r=0$ omit the residual argument; for $r=d$ use only the local parameter inverse. The lower bound is \eqref{eq:momentlower}. Approximating an arbitrary completion point $P$ by finite points proves the extension, after decreasing $\delta_E$ if necessary.
\end{proof}

For the uniform point $U$, the criterion and estimates simplify to
\begin{equation}\label{eq:uniformdistance}
 \begin{gathered}
 P_j\to U\ \Longleftrightarrow\ \ell(\Sigma_j)\to\infty,\\
 e^{-\ell(\Sigma)/2}\le d_{\mathrm W}((\mu,\Sigma),U)
       \le C_de^{-\ell(\Sigma)/2}
 \qquad(\ell(\Sigma)\ge t_d),
 \end{gathered}
\end{equation}
for dimension-dependent $t_d,C_d$. In particular, the lower bound on $\lambda_{\min}(\Sigma)$ in the usual uniform approximation can be replaced by the weaker lattice condition. The completion has exactly one point with uniform density.

\subsection{The pieces of the added set}\label{subsec:pieces}
For a primitive subgroup $L$, let $\mathcal S_L$ be the set of completion points with that precise surviving subgroup. Write $\mathcal M_0=\{U\}$ for the zero-dimensional model.
\begin{proposition}[Isometric pieces and their closures]\label{prop:pieces}
If $\operatorname{rank}L=r$, then $\mathcal S_L$, with distance inherited from the full completion, is isometric to the finite wrapped family $\mathcal M_r$ with its intrinsic Fisher distance. Its closure is isometric to $\widehat{\mathcal M}_{r,\mathrm W}$, and
\begin{equation}\label{eq:piececlosures}
 \overline{\mathcal S_L}=\bigcup_{L'\le L\ \mathrm{primitive}}\mathcal S_{L'},
 \qquad
 \overline{\mathcal S_L}\cap\overline{\mathcal S_\Gamma}
 =\overline{\mathcal S_{L\cap\Gamma}}.
\end{equation}
Primitivity in the union may equivalently be taken in $L$ or in $\Z^d$.
\end{proposition}
\begin{proof}
In adapted coordinates, adjoining an independent uniformizing factor gives the upper distance bound \eqref{eq:residuallift}. Marginalizing to the last $r$ observations gives the opposite bound. Indeed, a marginal score is the conditional expectation of the full score, so marginalization decreases Fisher speed and is $1$-Lipschitz for the corresponding intrinsic distances. This map extends to completions and sends $E(\beta,B)$ to $(\beta,B)$. Hence the inclusion is isometric. Its extension from the completion of $\mathcal M_r$ has complete, therefore closed, image, equal to $\overline{\mathcal S_L}$. The classification in dimension $r$ identifies its pieces as the primitive subgroups of $L$. Since $L$ is primitive in $\Z^d$, those are precisely the ambient primitive subgroups contained in $L$. The intersection of two primitive subgroups is primitive, and uniqueness of the classification proves the second equality.
\end{proof}

For $d=2$, Example~\ref{ex:d2} describes the whole added set: $U$ together with one closed disk for each primitive $n\in\Z^2$ up to sign, each isometric to the completed circular family by Proposition~\ref{prop:pieces}, Theorem~\ref{thm:completionclassification} with $d=1$, and Lemma~\ref{lem:localcirculardisk}. The topology is the $L^1$ subspace topology of Theorem~\ref{thm:completionclassification}; the pieces are not asserted to form a locally finite stratification.

\subsection{Covariance divergence and convergence examples}\label{subsec:divergence}
Every sequence converging to a proper-rank point has $\|\Sigma_j\|_{\mathrm{op}}\to\infty$. In adapted coordinates, $A_j\succeq S_j$ and each diagonal entry of $S_j$ is at least $\ell(S_j)\to\infty$; a fixed change of coordinates preserves divergence of the norm. Divergence of the covariance norm alone is not a sufficient convergence condition.

\begin{proposition}[Positive-semidefinite covariance rays]\label{prop:psdrays}
Let $\Sigma_0>0$ and $H\succeq0$, $H\ne0$. For fixed $\mu$, the curve $(\mu,\Sigma_0+tH)$, $0\le t<\infty$, has finite wrapped Fisher length. Its endpoint is specified by
\[
 L=\ker H\cap\Z^d,\qquad Q(n)=n^\top\Sigma_0n\ (n\in L),
 \qquad\chi(n)=e^{in^\top\mu}.
\]
In particular all curves $(\mu,\Sigma_0+tI_d)$ tend to $U$.
\end{proposition}
\begin{proof}
The density $F_t$ is the convolution of $F_0$ with the wrapped centered Gaussian probability law of covariance $tH$, possibly degenerate. Thus $F_t\ge c_0:=\min F_0>0$. Put $a_n=n^\top\Sigma_0n$ and $h_n=n^\top Hn$. Termwise differentiation of the Gaussian Fourier series and Tonelli's theorem give
\[
 \operatorname{Length}(F_t:0\le t<\infty)
 \le\frac1{2\sqrt{c_0}}\int_0^\infty
      \sum_{n:h_n>0}h_ne^{-a_n/2-th_n/2}\dd t
 =\frac1{\sqrt{c_0}}\sum_{n:h_n>0}e^{-a_n/2}<\infty.
\]
No positive lower bound on the nonzero $h_n$ is needed. Dominated convergence identifies the uniform density limit with the stated moments, and Theorem~\ref{thm:completionclassification} identifies its endpoint.
\end{proof}

For example, in dimension two, $H=vv^\top$ with $v=(1,1)$ preserves the frequency $(1,-1)$ and yields a rank-one point. With $v=(1,\sqrt2)$ it preserves no nonzero integer frequency and yields $U$. Thus the rank of the divergent real covariance does not by itself determine the dimension of the residual wrapped model.

There are also Fisher-convergent sequences with a collapsing least covariance eigenvalue. Let
\[
 \Sigma_j=j\begin{pmatrix}1+j^2&-j\\-j&1\end{pmatrix},\qquad\mu_j=0.
\]
The matrix in brackets is $R_jR_j^\top$ with $R_j=\left(\begin{smallmatrix}1&-j\\0&1\end{smallmatrix}\right)$ unimodular. Hence $\ell(\Sigma_j)=j$, and \eqref{eq:uniformdistance} implies convergence to $U$ at distance comparable to $e^{-j/2}$. Nevertheless $\lambda_{\min}(\Sigma_j)\sim j^{-1}$.

Finally, weak convergence is insufficient. Fix $a>0$ and put
\[
 M_j=\begin{pmatrix}1&j\\j&j^2+1\end{pmatrix}\in\mathrm{SL}(2,\Z),
 \qquad\Sigma_j=aM_j^\top M_j,\qquad\mu_j=0.
\]
For every fixed $n=(p,q)\ne0$, $|M_jn|\to\infty$, using the quadratic term if $q\ne0$ and the linear term otherwise. All fixed nonzero moments therefore tend to zero, which implies weak convergence to uniformity by density of trigonometric polynomials in the continuous functions on the torus. But $n_j=(j^2+1,-j)$ satisfies $M_jn_j=(1,0)$, so
\[
 |m_{n_j}(P_j)|=e^{-a/2},\qquad
 \|F_{P_j}-1\|_1\ge e^{-a/2},\qquad d_{\mathrm W}(P_j,U)\ge e^{-a/2}.
\]
This sequence is not Fisher-Cauchy: an $L^1$ limit would have to equal its weak limit, contradicting the displayed bound. Thus divergence at each fixed nonzero frequency cannot replace the uniform lattice condition in \eqref{eq:outsidecriterion}.

\section{Contractions and geometry between pieces}\label{sec:contractions}
\subsection{Heat contraction and a global deformation to uniformity}
\label{subsec:heatcontraction}

The completion is topologically simple in one global respect despite its
singular local behavior.  Let $k_t$ be the centered wrapped Gaussian density
of covariance $tI_d$, relative to normalized Haar measure, and write
\[
 \mathcal G_tF=k_t*F,\qquad
 b_t=\min_{\theta\in\T^d}k_t(\theta),\qquad t>0.
\]
Here convolution uses normalized Haar measure.  The spatial Gaussian sum
gives $b_t>0$, while normalization and the nonzero first Fourier coefficient
give $b_t<1$.

\begin{theorem}[Strict heat contraction and contractibility]
\label{thm:heatcontraction}
For $t\ge0$, the map
\[
 \Phi_t(\mu,\Sigma)=(\mu,\Sigma+tI_d),\qquad \Phi_0=\mathrm{id},
\]
extends uniquely to the wrapped Fisher completion.  These maps form a
jointly continuous semigroup, fix $U_d$, and satisfy, for every pair of
completion points $P,Q$ and every $t>0$,
\begin{equation}\label{eq:strictheatcontraction}
 d_{\mathrm W}(\Phi_tP,\Phi_tQ)
 \le\sqrt{1-b_t}\,d_{\mathrm W}(P,Q).
\end{equation}
Their densities are $\mathcal G_tF_P$.  In the classification
\eqref{eq:classifiedmoments},
\begin{equation}\label{eq:heatclassified}
 \Phi_t(L,\chi,Q)=(L,\chi,Q+t|\cdot|^2|_{\operatorname{span}_{\R}L}).
\end{equation}
With the constants in \eqref{eq:uniformdistance},
\begin{equation}\label{eq:uniformheatattraction}
 \sup_{P\in\widehat{\mathcal M}_{d,\mathrm W}}
 d_{\mathrm W}(\Phi_tP,U_d)\le C_de^{-t/2},\qquad t\ge t_d.
\end{equation}
Consequently the continuous homotopy
\begin{equation}\label{eq:heatdeformation}
 \mathcal H(P,s)=
 \begin{cases}
 \Phi_{s/(1-s)}P,&0\le s<1,\\
 U_d,&s=1
 \end{cases}
\end{equation}
is a strong deformation retraction of the full completion onto $U_d$.
It restricts to a strong deformation retraction of the added set, and of
each closed piece $\overline{\mathcal S_L}$, onto $U_d$.
\end{theorem}
\begin{proof}
For a fixed $t>0$, decompose the Markov operator as
\[
 \mathcal G_t=b_t\Pi+(1-b_t)\mathcal R_t,
 \qquad \Pi F=\int F\dd\nu,
\]
where $\mathcal R_t$ is convolution with the nonnegative probability density
$(k_t-b_t)/(1-b_t)$.  If $F$ is a positive density and $h$ is a real
density variation with integral zero, then $\mathcal G_th=(1-b_t)\mathcal R_th$ and
$\mathcal G_tF=b_t+(1-b_t)\mathcal R_tF$.  Therefore
\begin{align*}
 \int\frac{(\mathcal G_th)^2}{\mathcal G_tF}\dd\nu
 &\le(1-b_t)\int\frac{(\mathcal R_th)^2}{\mathcal R_tF}\dd\nu\\
 &\le(1-b_t)\int\frac{h^2}{F}\dd\nu.
\end{align*}
The first inequality uses $\mathcal G_tF\ge(1-b_t)\mathcal R_tF$; this Doeblin-type splitting is the mechanism behind strong data-processing inequalities \cite{PolyanskiyWu2017,Raginsky2016}. The second is the usual conditional Cauchy--Schwarz inequality
for a Markov kernel, the $\chi^2$ form of data processing: pointwise,
$(\mathcal R_th)^2\le (\mathcal R_tF)\mathcal R_t(h^2/F)$, and integration preserves mass.
For finite wrapped parameters, convolution with $k_t$ adds $tI_d$ to the
covariance, so it sends every admissible connecting curve to another such
curve.  Integrating the speed inequality proves
\eqref{eq:strictheatcontraction} for finite parameters.  The Lipschitz
extension theorem for metric completions gives the unique extension and
the same inequality there.  The semigroup identities follow by density,
as does the identity of its density with $\mathcal G_tF_P$.

Fourier multiplication gives
$m_n(\Phi_tP)=e^{-t|n|^2/2}m_n(P)$, proving
\eqref{eq:heatclassified} and $\Phi_tU_d=U_d$.  In particular every finite
time preserves the surviving subgroup exactly.  For finite $P$,
$\ell(\Sigma+tI_d)\ge t$, so \eqref{eq:uniformdistance} proves
\eqref{eq:uniformheatattraction}; approximation by finite points proves
it for every completion point, with the same constants.

At finite parameters, smoothness of the Fisher metric gives
$d_{\mathrm W}(\Phi_tP,P)\to0$ as $t\downarrow0$.  For a completion point
$P$ choose a finite $P'$ close to it.  Nonexpansiveness gives
\[
 d_{\mathrm W}(\Phi_tP,P)
 \le 2d_{\mathrm W}(P,P')+d_{\mathrm W}(\Phi_tP',P'),
\]
proving continuity at $t=0$ for every $P$.  Nonexpansiveness and the
semigroup law then give joint continuity in $(P,t)$.  More explicitly,
for $t\ge u\ge0$,
\[
 d_{\mathrm W}(\Phi_tP,\Phi_uQ)
 \le d_{\mathrm W}(P,Q)+d_{\mathrm W}(\Phi_{t-u}Q,Q).
\]
Uniform attraction proves continuity of \eqref{eq:heatdeformation} at
$s=1$, uniformly in $P$.  It starts at the identity, ends at $U_d$, and
fixes $U_d$ throughout.  By \eqref{eq:heatclassified}, finite times
preserve the added set and each piece; the final point belongs to the
closure of every piece.  This proves all the stated retractions.
\end{proof}

\begin{remark}\label{rem:contraction-scope}
The contraction is a statement about distances in the full wrapped
family and its completion.  It does not assert local compactness or the
existence of minimizing geodesics at the added points.  In particular,
contractibility is compatible with the obstruction in
Theorem~\ref{thm:noncompact}.  The isotropic heat operator here uses the
fixed observation coordinates; a general integer change of coordinates
conjugates it to the corresponding anisotropic Gaussian convolution.
\end{remark}

\begin{remark}[The contraction constant]\label{rem:heatconstant}
Put $q=e^{-t/2}$. By the Jacobi triple product
\cite[Eq.~20.5.3]{DLMF} (see also \cite{Mumford1983}), each one-dimensional factor of $k_t$ is
\[
 \sum_{n\in\Z}q^{n^2}e^{inu}
 =\prod_{j\ge1}(1-q^{2j})\,\bigl|1+q^{2j-1}e^{iu}\bigr|^2.
\]
Every factor is positive, even in $u$, and
nonincreasing for $u\in[0,\pi]$, so the minimum is attained at
$\theta=(\pi,\ldots,\pi)$ and
\[
 b_t=\Bigl(\sum_{n\in\Z}(-1)^ne^{-tn^2/2}\Bigr)^d
 =1-2d\,e^{-t/2}+O(e^{-t}),\qquad
 \sqrt{1-b_t}=\sqrt{2d}\,e^{-t/4}\bigl(1+O(e^{-t/2})\bigr).
\]
This uniform Lipschitz factor is not the large-time decay of the flow on
pairs of points: by \eqref{eq:uniformheatattraction} and the triangle
inequality, $d_{\mathrm W}(\Phi_tP,\Phi_tQ)\le2C_de^{-t/2}$ for all
completion points $P,Q$ and $t\ge t_d$. Inequality
\eqref{eq:strictheatcontraction} holds for every $t>0$ with a factor
strictly less than one; we do not determine the optimal Lipschitz
constant of $\Phi_t$.
\end{remark}

\subsection{Canonical retractions and local contractibility}
\label{subsec:retractions}

The isometric pieces admit a compatible family of retractions.  These
also show that failure of local compactness does not entail failure of
local contractibility.

\begin{theorem}[Retractions onto the closed pieces]
\label{thm:pieceretractions}
For every primitive subgroup $L\le\Z^d$, there is a canonical
$1$-Lipschitz retraction
\[
 R_L:\widehat{\mathcal M}_{d,\mathrm W}\longrightarrow
 \overline{\mathcal S_L}
\]
specified by
\begin{equation}\label{eq:retractionmoments}
 m_n(R_LP)=\begin{cases}m_n(P),&n\in L,\\0,&n\notin L.\end{cases}
\end{equation}
In particular,
\[
 R_L(\Gamma,\chi,Q)=(\Gamma\cap L,\chi|_{\Gamma\cap L},Q|_{\operatorname{span}_{\R}(\Gamma\cap L)}),
 \qquad R_LR_\Gamma=R_\Gamma R_L=R_{L\cap\Gamma}.
\]
The full completion strongly deformation retracts onto
$\overline{\mathcal S_L}$.  The retraction homotopy can be chosen
$1$-Lipschitz in its spatial variable at every time and preserves every
open or closed metric ball centered at a point of
$\overline{\mathcal S_L}$.
\end{theorem}
\begin{proof}
Choose adapted integer coordinates $(x,y)\in\T^{d-r}\times\T^r$,
where $r=\operatorname{rank}L$.  Marginalization to $y$ is
$1$-Lipschitz for intrinsic Fisher distances, as in the proof of
Proposition~\ref{prop:pieces}, and therefore extends to the metric
completions.  Compose it with the isometric inclusion that adjoins an
independent uniform $x$ factor.  The resulting map is a $1$-Lipschitz
retraction onto $\overline{\mathcal S_L}$.  Its moments are
\eqref{eq:retractionmoments}, proving independence of the adapted basis,
the classified formula, and the composition identities by uniqueness of
Fourier coefficients.

For the homotopy, let $\Phi_t^L$ add covariance
$t\diag(I_{d-r},0_r)$ in these adapted coordinates.  Convolution with
this possibly degenerate Gaussian is a Markov operator, so conditional
Cauchy--Schwarz makes $\Phi_t^L$ nonexpansive on the finite wrapped family.
It extends to a jointly continuous semigroup on the completion by the
same approximation argument as for $\Phi_t$.  Every point of
$\overline{\mathcal S_L}$ is fixed, as follows from its moments.
For each finite point $P$, Proposition~\ref{prop:psdrays} gives
\[
 \Phi_t^LP\longrightarrow R_LP\quad\text{in }d_{\mathrm W}
 \quad(t\to\infty).
\]
For an arbitrary completion point $P$ and any finite approximation $P'$,
nonexpansiveness of both maps gives
\[
 d_{\mathrm W}(\Phi_t^LP,R_LP)
 \le2d_{\mathrm W}(P,P')+
 d_{\mathrm W}(\Phi_t^LP',R_LP').
\]
It follows that this convergence holds for every $P$.  It is also joint
at infinite time: if $P_j\to P$ and $t_j\to\infty$, then
\[
 d_{\mathrm W}(\Phi_{t_j}^LP_j,R_LP)
 \le d_{\mathrm W}(P_j,P)+d_{\mathrm W}(\Phi_{t_j}^LP,R_LP)
 \longrightarrow0.
\]
Consequently
\[
 \mathcal H_L(P,s)=
 \begin{cases}\Phi^L_{s/(1-s)}P,&0\le s<1,\\R_LP,&s=1\end{cases}
\]
is the required strong deformation retraction.  If
$E\in\overline{\mathcal S_L}$, its fixed-point property and
nonexpansiveness give
$d_{\mathrm W}(\mathcal H_L(P,s),E)\le d_{\mathrm W}(P,E)$
at every time, proving the ball-preservation assertion.  For $r=d$ the
construction is the identity homotopy; for $r=0$ it retracts onto $U_d$.
\end{proof}

\begin{corollary}[Contractible small balls]
\label{cor:localcontractibility}
Every point of the wrapped Fisher completion has a neighborhood basis
of contractible open metric balls.  In particular the completion is
locally contractible, including at its non-locally-compact added points.
\end{corollary}
\begin{proof}
Fix a point $E$ and let $L$ be its precise surviving subgroup, of rank
$r$.  Under the isometry
$\overline{\mathcal S_L}\cong\widehat{\mathcal M}_{r,\mathrm W}$,
the point $E$ is a finite residual parameter.  If $r>0$, all its
coordinate moments $m_{e_i}(E)$ in that residual model are nonzero.
The moment Lipschitz inequality implies that any residual completion
point at distance less than
$\min_i|m_{e_i}(E)|$ still has all these moments nonzero.  Its surviving
subgroup contains an integer basis and is therefore the whole
$\Z^r$; thus it is a finite residual parameter.  Choose $\varepsilon_0>0$
smaller than this bound and sufficiently small that every residual
metric ball $B_r(E,\varepsilon)$, $0<\varepsilon<\varepsilon_0$, is a
convex normal ball for its smooth Fisher metric.

By Theorem~\ref{thm:pieceretractions}, the ambient ball
$B(E,\varepsilon)$ strongly deformation retracts, within itself, onto
\[
 B(E,\varepsilon)\cap\overline{\mathcal S_L}
 \cong B_r(E,\varepsilon).
\]
The latter contracts to $E$ by its radial minimizing geodesics, staying
within that ball and fixing $E$.  Concatenating the two homotopies
contracts the ambient ball.  For $r=0$ the first retraction already has
image $\{E\}=\{U_d\}$, so every ball about $U_d$ is contractible.
For finite full-dimensional points, the same argument reduces to the
usual local Riemannian statement.
\end{proof}

\subsection{Product slices, distances between pieces, and geodesics}
\label{subsec:transversegeometry}

Write $X_j=\widehat{\mathcal M}_{j,\mathrm W}$ and $d_j=d_{\mathrm W,j}$,
with $X_0=\{U_0\}$. Products below refer to independent observation
blocks in fixed integer coordinates. An integer torus automorphism
transports all the statements to any compatible primitive splitting.

\begin{proposition}[Product distance bounds and isometric slices]
\label{prop:productslices}
For $d=d_1+\cdots+d_k$, the independent-product map
$\mathcal I:\prod_iX_{d_i}\to X_d$ satisfies
\begin{equation}\label{eq:productdistancebounds}
 \max_i d_{d_i}(P_i,Q_i)
 \le d_d\bigl(\mathcal I(P_1,\ldots,P_k),
                    \mathcal I(Q_1,\ldots,Q_k)\bigr)
 \le\left(\sum_i d_{d_i}(P_i,Q_i)^2\right)^{1/2}.
\end{equation}
In particular, fixing every factor except one gives an isometric
embedding of the remaining completed model, with closed image.
For $P\in X_m$ and $Q\in X_r$, moreover,
\begin{equation}\label{eq:productprojectiondistance}
 \operatorname{dist}_{d_{m+r}}(P\otimes Q,U_m\otimes X_r)
 =d_m(P,U_m),
\end{equation}
and the distance is attained at $U_m\otimes Q$.
\end{proposition}
\begin{proof}
For finite parameters, independence makes the product Fisher speed
squared the sum of the factor speeds squared: the mixed terms vanish
because scores have mean zero. Parametrize factor curves of lengths
$\ell_i$ at constant speed on the same unit interval. Their product
has length $(\sum_i\ell_i^2)^{1/2}$. Infimizing over the factor
curves proves the upper bound. The marginal of a wrapped normal law on a block of integer coordinates is the wrapped normal law with the corresponding blocks of $\mu$ and $\Sigma$, because wrapping commutes with coordinate projection; this holds with arbitrary cross-covariance. Thus marginal curves stay in $\mathcal M_{d_i}$. Marginalization to each block decreases
Fisher speed and gives the lower bounds. The upper bound extends the
product map continuously to the product of the completions. The
extended density map identifies its value with the product density,
since products of probability densities converge in $L^1$ when their
factors do. Marginalization extends by continuity and retains the
lower bounds. With only one changing factor the two bounds agree.
Its image is complete, hence closed in $X_d$.
For \eqref{eq:productprojectiondistance}, marginalization to the
first block gives the lower bound against every $U_m\otimes Q'$;
the fixed-$Q$ isometric slice attains it at $Q'=Q$.
\end{proof}

Only elementary local circular geometry is needed for a geodesic
application. With normalized angular measure $\nu_1$, put
\[
 H_z(\theta)=1+2\operatorname{Re}\sum_{n\ge1}
 z^n(z\bar z)^{n(n-1)/2}e^{-in\theta},\qquad |z|<1.
\]
For $z=qe^{i\alpha}$ this is the circular wrapped density relative
to $\nu_1$, and $H_0=1$.

\begin{lemma}[A local circular disk at uniformity]
\label{lem:localcirculardisk}
The circular Fisher metric extends real analytically and positively
across $z=0$, with
$g_0=2(\dd x^2+\dd y^2)$ for $z=x+iy$.
Writing it in polar coordinates as
$B(q)\dd q^2+A(q)\dd\alpha^2$, define
\[
 R(q)=\int_0^q\sqrt{B(v)}\,\dd v,
 \qquad
 B(q)=\int\frac{(\partial_qH_q)^2}{H_q}\,\dd\nu_1.
\]
Then $d_1(H_{qe^{i\alpha}},U_1)=R(q)$ and
$R(q)=\sqrt2q(1+O(q^2))$. For all sufficiently small $q>0$,
the unique minimizing circular curve between
$H_{qe^{i\alpha}}$ and $H_{-qe^{i\alpha}}$, up to
parametrization, is the radial diameter through $U_1$; its length
is $2R(q)$.
\end{lemma}
\begin{proof}
The displayed series is real analytic in $(x,y)$ and normally
convergent with all derivatives on smaller disks. Near zero it is
bounded below, so integrating
$H_z^{-1}\dd H_z\otimes\dd H_z$ gives an analytic metric.
At zero its two scores are $2\cos\theta$ and $2\sin\theta$;
orthogonality gives $g_0=2I$. Rotations and reflection of the
observation circle imply the polar form. Since $H_{-q}(\theta)=H_q(\theta+\pi)$, the coefficient $B$ is even and analytic, so $B(q)=2+O(q^2)$.
Small circular arcs have length $O(q)$, so paths through zero
can be approximated in length by punctured-disk paths. Thus the
analytic local metric agrees with the completed circular distance.
Every curve has length at least its total radial variation in $R$,
whereas a radial curve has exactly that length. Passing to the
uniform limit gives the distance formula and its expansion.

By the first part, the center is an ordinary point of a smooth Riemannian disk. Radial lines in arclength are geodesics, being fixed sets of reflections, and continue through the
center as opposite radial lines. We use the standard convex-normal-neighborhood theorem for smooth Riemannian metrics \cite[Chapter~6]{Lee2018}. Choose a strongly convex
normal neighborhood of the center containing a coordinate disk
$|z|\le\delta$. If $R(q)<R(\delta)/2$, any connecting curve
leaving that coordinate disk has length at least
$2(R(\delta)-R(q))>2R(q)$. Inside the normal neighborhood the
short radial diameter is the unique minimizing geodesic. The exit
bound makes this a global distance statement in $X_1$, not merely
a comparison among curves restricted to the coordinate disk.
\end{proof}

\begin{theorem}[Minimizers through added laws and nonattainment]
\label{thm:addedpointgeodesics}
Every added point of $X_d$ is an interior point of a nonconstant
minimizing geodesic in $X_d$. More specifically, write the point in
adapted coordinates as $E=U_m\otimes P_r$, with $m=d-r\ge1$
and $P_r\in\mathcal M_r$. Then
\begin{equation}\label{eq:transversedisk}
 z\longmapsto E_z:=U_{m-1}\otimes H_z\otimes P_r
\end{equation}
is an isometric circular disk centered at $E$. For sufficiently
small $q>0$ and every $\alpha$, a minimizing curve from
$E_{qe^{i\alpha}}$ to $E_{-qe^{i\alpha}}$ passes through $E$
and has length $2R(q)$.

If $r=d-1$, both endpoints, and the punctured diameter between
them, belong to the finite family $\mathcal M_d$. Nevertheless,
no curve staying entirely in $\mathcal M_d$ attains the distance
between these endpoints. Thus $\mathcal M_d$ need not be a
geodesic metric space, although its completion contains an explicit
minimizer for these pairs.
\end{theorem}
\begin{proof}
The map in \eqref{eq:transversedisk} is a fixed-factor instance
of Proposition~\ref{prop:productslices}; its center has the density
of $E$, so it is the same completion point by the classification.
Apply Lemma~\ref{lem:localcirculardisk} and transport its diameter
through this isometry. For $m=1$, a nonzero circular parameter
and the fixed nondegenerate residual covariance give a finite
positive-definite block covariance, so only the center is added.

Suppose a finite-family curve attained the distance $2R(q)$.
Its marginal on the varying circular coordinate would have length
at most $2R(q)$ and the same circular endpoints. It must therefore
be the unique minimizing circular diameter, which passes through
$U_1$. But a finite Gaussian covariance has a finite marginal
variance, and its first circular moment has strictly positive
magnitude. Its marginal cannot be uniform. This contradiction
proves nonattainment in the finite family.
\end{proof}

The behavior changes for two independent transverse directions.
This gives a quantitative distance result between different pieces
and shows why their union is not a metric wedge at their intersection.

\begin{corollary}[A bridge between compatible pieces]
\label{cor:crosspiecebridge}
Suppose $m\ge2$ in the preceding theorem. In the same adapted
coordinates, for $q,s>0$ let
\[
 P=H_{qe^{i\alpha}}\otimes U_1\otimes U_{m-2}\otimes P_r,
 \qquad
 Q=U_1\otimes H_{se^{i\beta}}\otimes U_{m-2}\otimes P_r.
\]
Put $a(q)=\int\sqrt{H_q}\,\dd\nu_1$. Then
\begin{equation}\label{eq:crosspiecebounds}
 2\sqrt{2-2a(q)a(s)}
 \le d_d(P,Q)
 \le\sqrt{R(q)^2+R(s)^2}
 <d_d(P,E)+d_d(E,Q).
\end{equation}
In particular no minimizing curve between $P$ and $Q$ passes
through $E$. Uniformly in the phases and with no restriction on
the ratio $q/s$,
\begin{equation}\label{eq:crosspieceasymptotic}
 d_d(P,Q)=\sqrt{2(q^2+s^2)}
             \bigl(1+O(q^2+s^2)\bigr)
 \qquad(q,s\to0).
\end{equation}
The compatibility hypothesis means that the two transverse
frequency generators, together with a basis of the residual
subgroup, extend to an integer basis. Arbitrary primitive
directions need not satisfy it.
\end{corollary}
\begin{proof}
The upper bound is Proposition~\ref{prop:productslices}; the
distances to $E$ are $R(q)$ and $R(s)$ by its isometric-slice
assertion. The strict inequality follows because both are positive.
Any curve passing through $E$ has length at least their sum.

The square-root map $F\mapsto2\sqrt F$ takes a Fisher curve to
an $L^2(\nu)$ curve of the same speed. Its chord is therefore a
lower bound on the intrinsic completed distance. For the two
product densities above, the square-root affinity factors as
$a(q)a(s)$, since the common residual density integrates to one.
This proves the lower bound. Finally,
$a(q)=1-q^2/4+O(q^4)$: expand the square root of
$H_q=1+2q\cos\theta+O(q^4)$ and integrate; odd powers vanish
by the symmetry $H_{-q}(\theta)=H_q(\theta+\pi)$.
Together with the expansion for $R$, the two bounds give
\eqref{eq:crosspieceasymptotic}. All remainders are uniform on a
fixed small parameter disk and are independent of the phases and
of the common residual factor.
\end{proof}

These statements settle explicit families of cross-piece distances
and minimizing curves through added laws. They do not assert
existence of minimizing geodesics for every pair in the nonproper
completion, nor do they give distances between arbitrary
incompatible frequency pieces.

\section{Local topology at the added points}\label{sec:topology}
\subsection{Integer shears and failure of local compactness}
\begin{lemma}[Torus isometries fixing the endpoint]\label{lem:torusisometries}
For each $T\in\mathrm{GL}(d,\Z)$, the transformation
\[
 \tau_T(\mu,\Sigma)=(T^{-1}\mu,T^{-1}\Sigma T^{-\top})
\]
is an isometry of the full wrapped model. Its extension to the completion fixes $U$.
\end{lemma}
\begin{proof}
The map $\theta\mapsto T^{-1}\theta$ is a parameter-independent diffeomorphism of the torus, preserves normalized Haar measure, and transforms the wrapped law exactly as displayed. Changing sample variables therefore preserves all Fisher pairings. Its inverse is induced by $T^{-1}$, so intrinsic distances are preserved in both directions and the isometry extends to the completion. The isometry preserves the uniform limiting density. Injectivity of the extended density map in Theorem~\ref{thm:completionclassification} therefore implies that its extension fixes $U$.
\end{proof}

\begin{theorem}[Failure of local compactness at every added point]\label{thm:noncompact}
For every $d\ge2$, the full wrapped metric completion $\widehat{\mathcal M}_{d,\mathrm W}$ is not locally compact at any point of $\widehat{\mathcal M}_{d,\mathrm W}\setminus\mathcal M_d$. More precisely, every ball about each added point contains infinitely many interior points separated by a positive distance. No added point has a neighborhood homeomorphic to a finite-dimensional topological manifold, with or without boundary.
\end{theorem}
\begin{proof}
For integers $k\ge0$ set $T_k=I_d+k\,e_1e_2^\top$, whose only off-diagonal nonzero entry is $k$, in position $(1,2)$. Fix $a\ge t_d$, put $q=e^{-a/2}$, and let
\begin{equation}\label{eq:shearpoints}
 \Sigma_k=aT_k^{-1}T_k^{-\top},\qquad P_k=(0,\Sigma_k).
\end{equation}
All covariances are positive definite. Their first two-coordinate block is
\[
 a\begin{pmatrix}1+k^2&-k\\-k&1\end{pmatrix},
\]
and the remaining block is $aI_{d-2}$ when $d>2$. By Lemma~\ref{lem:torusisometries} and \eqref{eq:uniformdistance},
\begin{equation}\label{eq:shearupper}
 d_{\mathrm W}(P_k,U)=d_{\mathrm W}(P_0,U)\le C_dq
 \quad\hbox{for every }k.
\end{equation}
The isometry fixes $U$; equivalently, $\ell(\Sigma_k)=a$ for every $k$. The least Euclidean eigenvalues are not bounded below uniformly in $k$.

For $n_k=T_k^\top e_1=(1,k,0,\ldots,0)$, direct multiplication gives
\[
 T_j^{-\top}n_k=(1,k-j,0,\ldots,0),\qquad
 m_{n_k}(P_j)=q^{1+(k-j)^2}.
\]
For $j\ne k$, Lemma~\ref{lem:observables} consequently yields
\begin{equation}\label{eq:shearseparation}
 d_{\mathrm W}(P_k,P_j)
 \ge q-q^{1+(k-j)^2}\ge q(1-q)>0.
\end{equation}
This lower bound applies to every connecting path in the full model, including paths that leave any fixed diagonal or sheared subfamily.

Given $\varepsilon>0$, choose a single sufficiently large $a$ so that $C_dq<\varepsilon$. The sequence $(P_k)$ lies in the $\varepsilon$-ball about $U$, and its separation is fixed and positive for that choice of $a$.

Now let an added point have residual rank $1\le r<d$, put $m=d-r$, and choose adapted coordinates. Translating observations reduces its residual mean to zero. Its density is then $U_m\otimes\mathrm{WN}_r(0,B)$. Starting with covariance $\diag(aI_m,B)$, apply the integer shears
\[
 (x_1,x_2,\ldots,x_m,y)\longmapsto(x_1-ky_1,x_2,\ldots,x_m,y),
 \qquad k\ge0,
\]
and let $P_k$ be the resulting mean-zero finite parameters. Their residual covariance is $B$ and their Schur complement is $aI_m$. Proposition~\ref{prop:endpointdistance} gives
\[
 d_{\mathrm W}(P_k,E(0,B))\le C_{B,m}e^{-a/2}
\]
independently of $k$. At the integer frequency $n_k=e_{x_1}+ke_{y_1}$, with $q=e^{-a/2}$,
\[
 m_{n_k}(P_j)=q\exp\!\left(-\tfrac12(k-j)^2B_{11}\right),\qquad
 d_{\mathrm W}(P_k,P_j)\ge q(1-e^{-B_{11}/2})>0\quad(j\ne k).
\]
Thus each prescribed ball about this endpoint also contains an infinite separated set, after choosing $a$ once for the ball. Integer coordinate changes and translations cover all classified added points.

If any added point had a compact neighborhood, some ball about it would lie in that neighborhood. A compact metric space is totally bounded and cannot contain the separated set just constructed. Local compactness is necessary for a finite-dimensional manifold neighborhood, including at a boundary point, proving the final assertion.
\end{proof}

\begin{corollary}[Infinite packing at a fixed relative scale]\label{cor:infinitepacking}
Let $d\ge2$ and let $E$ be any added point. There are constants
$R_E>0$ and $c_E\in(0,1)$ such that every ball $B(E,R)$ with
$0<R<R_E$ contains infinitely many finite-parameter points with pairwise
distances at least $c_E R$. In particular, no neighborhood of $E$ is
doubling or has finite Assouad dimension, and no such neighborhood has a
bi-Lipschitz embedding into a finite-dimensional Euclidean space
\cite{Heinonen2001,Fraser2020}.
\end{corollary}
\begin{proof}
In the proof of Theorem~\ref{thm:noncompact}, the constants can be fixed
before choosing $a$. More precisely, for all sufficiently small
$q=e^{-a/2}$ the constructed points satisfy
\[
 d_{\mathrm W}(P_k,E)\le C_E q,\qquad
 d_{\mathrm W}(P_k,P_j)\ge b_E q\quad(k\ne j),
\]
where $C_E\ge1$ is fixed. At $E=U$, take $b_E=1/2$ and restrict to
$q\le1/2$. At a positive residual rank, take
$b_E=1-e^{-B_{11}/2}>0$ in adapted coordinates. Given sufficiently
small $R$, choose $q=R/(2C_E)$. All the points lie in $B(E,R)$ and
their pairwise distances are at least $c_E R$, with
$c_E=b_E/(2C_E)$.

A ball of radius less than $c_E R/2$ can contain at most one of these
points. Hence $B(E,R)$ cannot be covered by finitely many balls at that
relative scale. Iterating a finite doubling bound would give such a
cover, which is impossible. Likewise any finite Assouad bound
$N_\rho(B(x,R))\le C(R/\rho)^s$ would give a finite cover at that scale.
Here $N_\rho$ is the least number of radius-$\rho$ balls needed for a
cover; we allow $N_\rho=\infty$. Finally, under a bi-Lipschitz embedding,
these points would give an infinite uniformly separated bounded set in
Euclidean space, contradicting its total boundedness on bounded sets.
\end{proof}

\begin{proposition}[Finite Hausdorff dimension despite infinite local packing]\label{prop:hausdorffdimension}
Let $\partial_{\mathrm W}\mathcal M_d=
\widehat{\mathcal M}_{d,\mathrm W}\setminus\mathcal M_d$ denote the added
set, without asserting that it is a manifold boundary. With the completed
Fisher distance,
\begin{equation}\label{eq:hausdorffdimensions}
 \dim_{\mathrm H}\widehat{\mathcal M}_{d,\mathrm W}
 =\frac{d(d+3)}2,\qquad
 \dim_{\mathrm H}\partial_{\mathrm W}\mathcal M_d
 =\frac{(d-1)(d+2)}2.
\end{equation}
Every neighborhood of any added point has the first Hausdorff dimension,
and its intersection with the added set has the second. For $d\ge2$,
the same neighborhood has infinite Assouad dimension.
\end{proposition}
\begin{proof}
Put $n_r=r(r+3)/2$ for $r\ge0$. There are countably many primitive
subgroups of $\Z^d$, since each has a finite integer basis.
Proposition~\ref{prop:pieces} identifies the piece for a rank-$r$
subgroup isometrically with the smooth Riemannian manifold
$\mathcal M_r$, of real dimension $n_r$. A smooth positive Riemannian
metric is locally bi-Lipschitz to Euclidean distance on smaller coordinate
neighborhoods; hence that piece has Hausdorff dimension $n_r$. The
rank-zero piece is a singleton. Hausdorff dimension is stable under
countable unions, so the dimensions of the whole completion and its
added set are respectively $\max_{0\le r\le d}n_r=n_d$ and
$\max_{0\le r<d}n_r=n_{d-1}$.

Let $V$ be a neighborhood of an added point $E$. The finite family is
dense in its completion, so $V$ contains a nonempty relatively open
subset of $\mathcal M_d$ and has dimension at least $n_d$; the global
upper bound gives equality. If $E$ has surviving subgroup $L$, extend
a basis of $L$ to a primitive subgroup $\Gamma$ of rank $d-1$.
Equation~\eqref{eq:piececlosures} gives $E\in\overline{\mathcal S_\Gamma}$.
Thus $V\cap\mathcal S_\Gamma$ contains a nonempty open subset of the
rank-$(d-1)$ piece and has dimension $n_{d-1}$. For $d=1$, take
$\Gamma=\{0\}$, and this statement is the zero-dimensional singleton case.
The final claim is Corollary~\ref{cor:infinitepacking}.
\end{proof}

For $d=2$, these dimensions are five for the completion and two for
the added set. The infinite Assouad dimension records a failure of
uniform covering at small scales, and does not imply infinite Hausdorff
dimension or an infinite-dimensional parameter manifold. The distinction
uses the intrinsic completed distance throughout. For the metric-space
formulation of Assouad dimension and its relation to doubling, see
\cite[Section~2.2]{Ishiki2021}; the packing contradiction above is direct.

\begin{corollary}[Obstruction to a smooth Fisher extension]\label{cor:nosmooth}
For $d\ge2$, no added point has a neighborhood in $\widehat{\mathcal M}_{d,\mathrm W}$ admitting a finite-dimensional smooth manifold structure, with or without boundary, compatible with the completion topology. In particular, no such neighborhood carries a positive-definite smooth Riemannian metric extending $g_{\mathrm W}$. For each added point $E$, the same conclusions hold for $\mathcal M_d\cup\{E\}$ with its distance inherited from the completion.
\end{corollary}
\begin{proof}
Every finite-dimensional manifold, including one with boundary, has compact coordinate neighborhoods. This contradicts Theorem~\ref{thm:noncompact}. A positive-definite smooth Riemannian metric induces the manifold topology. If only $E$ is adjoined, each ball about it still contains the separated interior points from that theorem, giving the same contradiction.
\end{proof}

\begin{corollary}[The same sequence in the two geometries]\label{cor:shearcontrast}
For the parameters in \eqref{eq:shearpoints},
\[
 d_{\mathrm Q}(P_0,P_k)\ge\frac1{\sqrt2}\log(1+k^2)\longrightarrow\infty,
 \qquad d_{\mathrm W}(P_0,P_k)\le2C_dq.
\]
The second bound can be made arbitrarily small by increasing $a$, while for each fixed $a$ the wrapped sequence remains positively separated.
\end{corollary}
\begin{proof}
Apply the logarithmic quadratic-form estimate in Proposition~\ref{prop:fullcomplete} with $x=e_1$: the ratio of final to initial variance is $1+k^2$. The wrapped estimate follows from the triangle inequality through $U$ in the completion. The original space embeds isometrically in its metric completion, so this is also a bound for its original intrinsic distance. Separation is \eqref{eq:shearseparation}.
\end{proof}

The theorem uses unrestricted covariance and a fixed torus; the points $P_k$ are distinct distributions, as their moments demonstrate. An additional quotient by integer torus automorphisms would identify these points and would be a different problem. When arbitrary covariances are allowed, every added point has the obstruction to a compatible smooth Riemannian extension.

\subsection{The integer-shear construction in two observation dimensions}\label{subsec:shear2}
The obstruction already occurs for $d=2$, where the full parameter space has dimension five: two means and three covariance entries. Fix $a\ge t_2$, set $q=e^{-a/2}$, and take independent real variables $Z_1,Z_2\sim N(0,a)$. For each integer $k\ge0$, put
\[
 X^{(k)}=(Z_1-kZ_2,Z_2),\qquad
 T_k=\begin{pmatrix}1&k\\0&1\end{pmatrix},\qquad
 \Sigma_k=a\begin{pmatrix}1+k^2&-k\\-k&1\end{pmatrix}.
\]
Then $X^{(k)}=T_k^{-1}(Z_1,Z_2)$ has covariance $\Sigma_k$, and wrapping gives the points $P_k$ of \eqref{eq:shearpoints}. The torus coordinates remain fixed as $k$ varies, so these are different distributions on the same sample space.

\begin{figure}[!htbp]
\centering
\includegraphics[width=\textwidth]{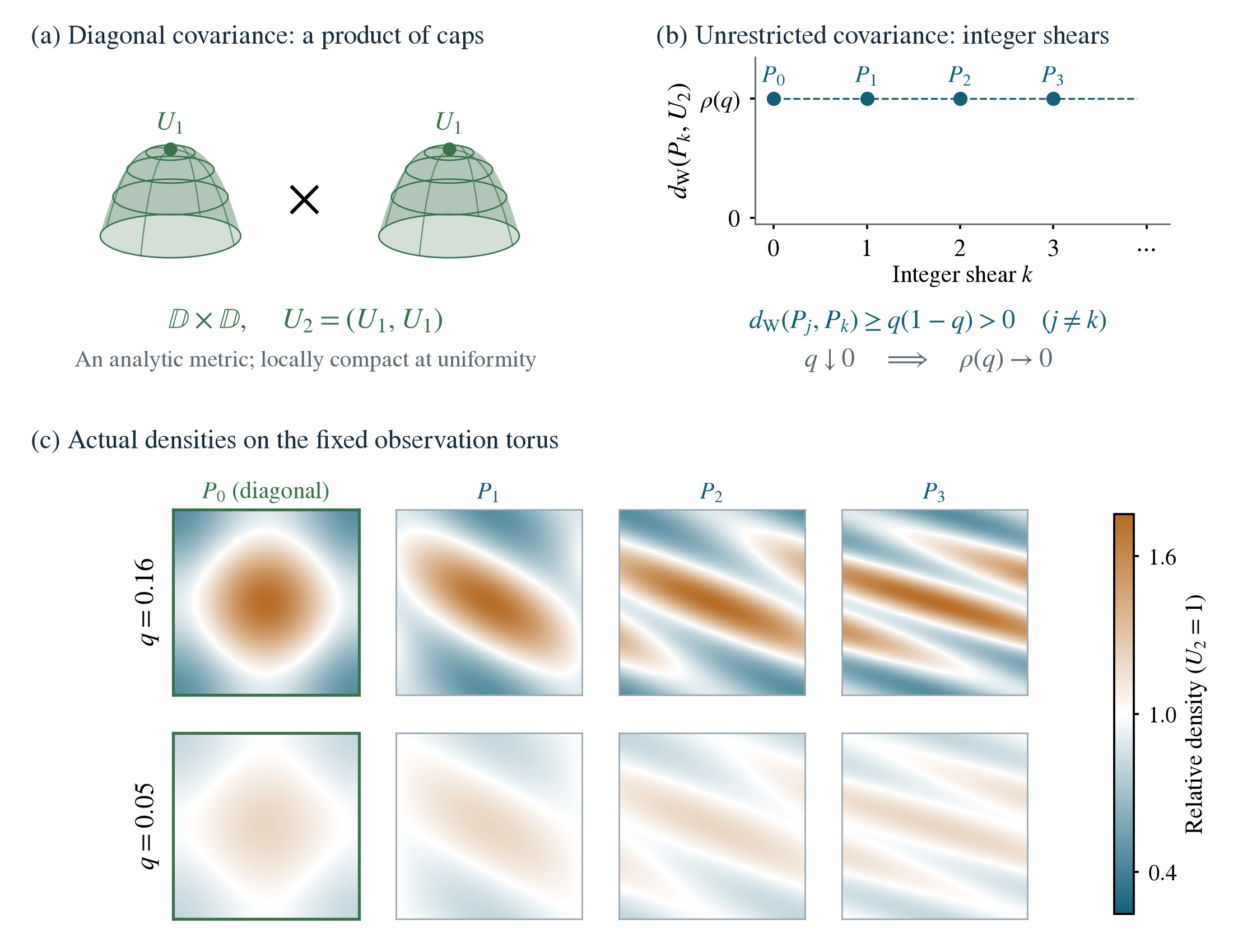}
\caption{Diagonal and unrestricted covariance in two observation dimensions. (a)~For comparison: with its own intrinsic distance, the diagonal subfamily is the Riemannian product of two circular families, since independence makes Fisher speeds add, and each factor completes to a disk; the upward-closing surfaces are schematic, not isometric embeddings. (b)~The full family contains the integer-shear sequence $P_k$ at a common distance from $U_2$, with the separation bound \eqref{eq:shearseparation}. The vertical scale is symbolic. (c)~Relative densities of $P_0,\ldots,P_3$ on $[-\pi,\pi]^2$ with opposite edges identified. Each row fixes $q$; only $P_0$ has diagonal covariance. Both rows use the same color scale, with white denoting the uniform density. Smaller $q$ brings the entire sequence closer to $U_2$.}\label{fig:shear-sequence}
\end{figure}

Their determinant, correlation coefficient before wrapping, and eigenvalues are
\begin{align}
 \det\Sigma_k&=a^2,&
 \rho_k&=-\frac{k}{\sqrt{1+k^2}},\label{eq:shearshape}\\
 \lambda_\pm(\Sigma_k)
 &=\frac a2\left(k^2+2\pm k\sqrt{k^2+4}\right).\label{eq:sheareigen}
\end{align}
These follow from the determinant and trace of the displayed matrix. In particular, as $k\to\infty$ at fixed $a$, $\lambda_+\sim ak^2$, $\lambda_-\sim a/k^2$, and $\rho_k\to-1$. Thus the least eigenvalue tends to zero along this sequence. Its uniform bound on distance to $U$ comes from torus isometries fixing $U$, not from applying the large-covariance estimate directly to every $\Sigma_k$.

The density of $P_k$ can also be written explicitly. With $F(\theta;q)=h_q(\theta)$ from Section~\ref{sec:fullcov}, its density relative to normalized Haar measure on $\T^2$ is
\begin{equation}\label{eq:sheardensity}
 F_k(\theta_1,\theta_2)
 =F(\theta_1+k\theta_2;q)F(\theta_2;q).
\end{equation}
Indeed, applying $T_k$ to the wrapped vector recovers the two independent wrapped $Z_j$; the torus map has Jacobian determinant one. The first factor depends on $\theta_1+k\theta_2$, while the second factor is independent of $k$; the Fisher distance to $U$ is the same for every $k$. For any integer vector $(m,n)$, the moment is
\begin{equation}\label{eq:shearallmoments}
 \E_{P_k}e^{i(mY_1+nY_2)}=q^{m^2+(n-km)^2},
\end{equation}
by the characteristic function of $mZ_1+(n-km)Z_2$. In particular the bounded observable $e^{i(Y_1+kY_2)}$ has expectation $q$ under $P_k$, but expectation $q^{1+(k-j)^2}\le q^2$ under every $P_j$ with $j\ne k$. Its frequency depends on $k$, while its absolute value is always one; Lemma~\ref{lem:observables} therefore gives the same separation bound $q(1-q)$ for all pairs.

At this fixed $a$, increasing $k$ does not make $P_k$ converge to $U$. The second marginal is unchanged and $m_{(0,1)}(P_k)=q$, whereas $m_{(0,1)}(U)=0$. In fact,
\[
 q\le d_{\mathrm W}(P_k,U)=d_{\mathrm W}(P_0,U)\le C_2q.
\]
Increasing $a$ places the whole separated sequence closer to $U$; increasing $k$ produces more separated points at that same distance. This distinction is essential to the local-compactness argument.

To exhibit the failure in a prescribed ball of radius $\varepsilon$, choose $a$ with $C_2e^{-a/2}<\varepsilon$ and then vary only $k$: every $P_k$ lies in the ball about $U$, but no subsequence is Cauchy. On the quotient side the variance of the first real coordinate grows from $a$ to $a(1+k^2)$, which gives the divergence in Corollary~\ref{cor:shearcontrast}.

\begin{remark}[Changing the period lattice]\label{rem:periodlattice}
For any fixed full-rank lattice $\Lambda=W\Z^d$, $W\in\mathrm{GL}(d,\R)$, the sample transformation $x\mapsto2\pi W^{-1}x$ identifies $\R^d/\Lambda$ with the torus used here. It transforms means and all positive-definite covariances bijectively and preserves both Fisher metrics by change of sample variables. The full-covariance comparison, completion classification, and nonlocal-compactness assertions therefore also hold for that fixed lattice, with the dual lattice $\Lambda^*=\{\xi:\xi\cdot\lambda\in2\pi\Z\ \text{for all}\ \lambda\in\Lambda\}$ in place of $\Z^d$ as the frequency lattice. The map also preserves $L^1$ distances between densities relative to normalized Haar measure, so the statistical results of Section~\ref{sec:statistics} transfer as well. Under a fixed change of lattice, the class $\Sigma\succeq\varepsilon I_d$ is carried between two such classes with constants depending on $W$, so the regularization result transfers with altered constants; what is not invariant is uniformity over integer automorphisms, since the images of $\mathcal M_d(\varepsilon)$ under the integer shears share no common spectral floor. The compactness criterion of Proposition~\ref{prop:compactcriterion} is stated through the frequency lattice and is invariant under both kinds of change.
\end{remark}

\section{Statistical obstruction and spectral regularization}\label{sec:statistics}
\begin{corollary}[Local minimax obstruction at every added point]\label{cor:nouniformestimation}
Let $d\ge2$ and let $E$ be an added point. Let $\mathfrak F$ be the
probability densities on $\T^d$, modulo equality $\nu$-almost everywhere,
with the $L^1$ metric and its Borel $\sigma$-algebra. There are
$R_E>0$ and $\lambda_E\in(0,1)$ such that, for every $0<R<R_E$,
every $n\ge1$, and every measurable estimator
$\widehat F_n:(\T^d)^n\to\mathfrak F$,
\begin{equation}\label{eq:localminimaxprob}
 \sup_{P\in B(E,R)\cap\mathcal M_d}
 P^{\otimes n}\!\left(\|\widehat F_n-F_P\|_1\ge\lambda_E R\right)=1.
\end{equation}
Moreover, the local minimax expected loss satisfies
\begin{equation}\label{eq:localminimaxrisk}
 \lambda_E R\le
 \inf_{\widehat F_n}\sup_{P\in B(E,R)\cap\mathcal M_d}
 \E_{P^{\otimes n}}\|\widehat F_n-F_P\|_1
 \le R.
\end{equation}
The same statements hold for $d_{\mathrm W}$ loss and measurable
estimators with values in $\widehat{\mathcal M}_{d,\mathrm W}$.
Consequently no estimator is uniformly consistent on any fixed
neighborhood of an added point. By contrast, a single moment estimator
is uniformly consistent on every compact subset of $\mathcal M_d$.
\end{corollary}
\begin{proof}
Use the constants $C_E\ge1$ and $b_E>0$ from the proof of
Corollary~\ref{cor:infinitepacking}. Choose $R_E$ so small that the
construction there applies to $q=R/(2C_E)$ whenever $0<R<R_E$.
The resulting shear points satisfy
\[
 d_{\mathrm W}(P_k,E)\le R/2,\qquad
 \|F_{P_k}-F_{P_j}\|_1\ge b_Eq\quad(k\ne j),
\]
where the second inequality follows from the same trigonometric moments
used in Theorem~\ref{thm:noncompact}. Put
$\lambda_E=b_E/(4C_E)$, so their $L^1$ separation is at least
$2\lambda_E R$.

Let $r$ be the residual rank of $E$ and put $m=d-r$.
In adapted coordinates the unsheared law is
$F_{P_0}(x,y)=H_q(x)F_E(y)$, where
$H_q(x)=\prod_{i=1}^{m}h_q(x_i)$; for $r=0$ take $F_E=1$.
Every shear used above preserves $\nu$ and the law $E$. Orthogonality
of the Fourier series therefore gives the exact identity
\begin{equation}\label{eq:shearchisquare}
 1+\chi^2(P_k\|E)
 =\int\frac{F_{P_k}^2}{F_E}\,\dd\nu
 =\int H_q^2\,\dd\nu_m
 =\mathcal T(q)^m,\qquad
 \mathcal T(q):=\sum_{j\in\Z}q^{2j^2}.
\end{equation}
The residual density cancels, so this quantity is independent of the
shear index and of the residual covariance. Here $E$ has a positive
density and is used only as a reference probability law.

Fix $J\ge2$ and consider $P_0,\ldots,P_{J-1}$. Under
$Q=E^{\otimes n}$ let
$L_k=\dd P_k^{\otimes n}/\dd Q$. Then
$\E_Q L_k^2=\mathcal T(q)^{mn}$. The success events
\[
 A_k=\{\|\widehat F_n-F_{P_k}\|_1<\lambda_E R\}
\]
are pairwise disjoint. Cauchy--Schwarz, first for each event and then
for the finite sum, yields
\[
 \frac1J\sum_{k=0}^{J-1}P_k^{\otimes n}(A_k)
 \le\frac{\mathcal T(q)^{mn/2}}{J}\sum_{k=0}^{J-1}Q(A_k)^{1/2}
 \le\frac{\mathcal T(q)^{mn/2}}{\sqrt J}.
\]
Thus the construction gives the explicit finite-hypothesis bound
\begin{equation}\label{eq:finitepackingrisk}
 \max_{0\le k<J}P_k^{\otimes n}
 \!\left(\|\widehat F_n-F_{P_k}\|_1\ge\lambda_E R\right)
 \ge1-\frac{\mathcal T(q)^{mn/2}}{\sqrt J}
 \ge1-\frac{\exp\{mnq^2/(1-q^2)\}}{\sqrt J}.
\end{equation}
For the last inequality use
$\mathcal T(q)\le1+2q^2/(1-q^2)$ and $\log(1+x)\le x$.
Sending $J\to\infty$ at fixed $R,n$ proves
\eqref{eq:localminimaxprob}. Multiplication by $\lambda_E R$
gives the lower bound in \eqref{eq:localminimaxrisk}. The constant
estimator $F_E$ gives its upper bound, since
$\|F_P-F_E\|_1\le d_{\mathrm W}(P,E)<R$.

The density map on the completion is $1$-Lipschitz by
Lemma~\ref{lem:observables} and
Theorem~\ref{thm:completionclassification}. Composing with it proves
the probability and expected-loss lower bounds for completion-valued
estimators in $d_{\mathrm W}$; the constant estimator $E$ again gives
the upper bound. Every neighborhood contains a ball as above, proving
the asserted obstruction on fixed neighborhoods.

For the positive assertion, use the empirical moments at $e_i$ and
$e_i+e_j$, $i<j$, and reconstruct the mean and covariance as in
Proposition~\ref{prop:fullregular}. Whenever a required moment is zero
or the reconstructed covariance is not positive definite, output the uniform density.
This defines one measurable estimator, independently of a compact set
$K\subset\mathcal M_d$. The moment inverse is continuous near the
compact image of $K$, with positive-definite reconstructed covariance.
For each $\varepsilon>0$, compactness therefore supplies $\eta>0$
such that errors below $\eta$ in all these moments give both $L^1$
and $d_{\mathrm W}$ errors below $\varepsilon$, uniformly on $K$.
There are $N=d(d+1)/2$ complex moments, each with empirical
mean-square error at most $1/n$. The union bound gives
\[
 \sup_{P\in K}P^{\otimes n}
 \!\left(\max_{\ell}|\widehat m_{n_\ell}-m_{n_\ell}(P)|\ge\eta\right)
 \le\frac{N}{n\eta^2}\longrightarrow0.
\]
This proves uniform consistency on every compact finite-parameter set.
\end{proof}

The obstruction is thus quantitative: on each sufficiently small
fixed ball about $E$, the minimax expected error has order $R$ for
every sample size. Infinite packing alone would not prove this;
\eqref{eq:shearchisquare} also controls the likelihood ratios of all
the separated hypotheses against one reference law. This direct
argument complements the information-theoretic lower-bound methods
of \cite{LeCam1973,Tsybakov2009,Yu1997} and the classical nonexistence
results of Bahadur and Savage \cite{BahadurSavage1956}.
The qualification ``fixed'' is essential. If $R=R_n\to0$ and the
center $E$ is known, the constant estimator $E$ already has error at
most $R_n$, so the result does not rule out consistency on shrinking
neighborhoods.

For statistical readers we stress what a fixed Fisher ball about an added point is in parameter terms. By \eqref{eq:uniformdistance}, the ball $B(U_d,R)$ contains every $(\mu,\Sigma)$ with $\ell(\Sigma)\ge t_d$ and $C_de^{-\ell(\Sigma)/2}<R$, where $\ell(\Sigma)=\min_{n\ne0}n^\top\Sigma n$ is the lattice minimum: the mean is arbitrary, and neither the condition number nor the smallest eigenvalue of $\Sigma$ is controlled. The shear covariances $\Sigma_k$, with eigenvalues of order $ak^2$ and $a/k^2$, all lie in one such ball. The same holds at the other added points, with $\ell$ replaced by the lattice minimum of the Schur complement in Theorem~\ref{thm:completionclassification} and the residual parameters close to those of the added point (Proposition~\ref{prop:endpointdistance}). Such a ball is therefore not a bounded or well-conditioned parameter region: it is a neighborhood of the limiting law, whose size is measured through the laws rather than through the parameters. Corollary~\ref{cor:nouniformestimation} says that uniform estimation fails on these neighborhoods, while Theorem~\ref{thm:spectralfloor} below shows that removing the poorly conditioned directions with $\lambda_{\min}(\Sigma)\to0$, but not the large eigenvalues, already restores it.

\subsection{A spectral floor restores uniform estimation}
\label{subsec:spectralfloor}

The positive estimator below uses the standard orthogonal-series construction \cite[Chapter~1]{Tsybakov2009}. The preceding obstruction uses covariance directions whose least Euclidean
eigenvalue tends to zero. A fixed lower spectral bound removes it, even
without an upper bound on covariance. The bound is imposed in the fixed
observation coordinates of the torus.

\begin{theorem}[Compact regularization and uniform estimation]
\label{thm:spectralfloor}
Fix $\varepsilon>0$ and put
\[
 \mathcal M_d(\varepsilon)
 =\{(\mu,\Sigma)\in\mathcal M_d:\Sigma\succeq\varepsilon I_d\},
 \qquad
 \mathcal K_\varepsilon
 =\overline{\mathcal M_d(\varepsilon)}^{\,d_{\mathrm W}}.
\]
Here $d_{\mathrm W}$ is the distance inherited from the full completion.
Then the following hold.
\begin{enumerate}
\item $\mathcal K_\varepsilon$ is compact. In the classification of
Theorem~\ref{thm:completionclassification}, it consists exactly of the
points $(L,\chi,Q)$ satisfying
\begin{equation}\label{eq:spectralfloorclosure}
 Q(v)\ge\varepsilon|v|^2
 \qquad(v\in\operatorname{span}_{\R}L).
\end{equation}
The condition is vacuous at $L=\{0\}$.
\item There are density-valued estimators $\widetilde F_n$ such that
\begin{equation}\label{eq:spectralfloorrate}
 \sup_{P\in\mathcal M_d(\varepsilon)}
 \E_{P^{\otimes n}}\|\widetilde F_n-F_P\|_1
 \le C_{\varepsilon,d}\,
 \frac{\{\log(n+1)\}^{d/4}}{\sqrt n}.
\end{equation}
The estimators and an explicit finite-cutoff bound are given below.
There are also estimators $\widehat P_n$ with values in
$\mathcal M_d(\varepsilon)$ attaining the same $L^1$ rate and satisfying
\begin{equation}\label{eq:spectralfloorfisher}
 \sup_{P\in\mathcal M_d(\varepsilon)}
 \E_{P^{\otimes n}}d_{\mathrm W}(\widehat P_n,P)
 \longrightarrow0.
\end{equation}
\end{enumerate}
No upper bound on $\|\Sigma\|_{\mathrm{op}}$ is imposed, and no optimality
of the logarithmic factor in \eqref{eq:spectralfloorrate} is asserted.
\end{theorem}
\begin{proof}
The Fourier moments satisfy the common envelope
\begin{equation}\label{eq:spectralenvelope}
 |m_k(P)|\le e^{-\varepsilon|k|^2/2},\qquad k\in\Z^d.
\end{equation}
Every sequence therefore has a subsequence whose Fourier coefficients
converge at every $k$. The summability of
$|k|^j e^{-\varepsilon|k|^2/2}$ for each $j\ge0$ makes the corresponding
densities and all their derivatives converge uniformly. Their limits
are probability densities; indeed $F_P\ge b_\varepsilon>0$, where
$b_\varepsilon$ is the minimum of the wrapped Gaussian kernel of covariance
$\varepsilon I_d$, because each $F_P$ is that kernel convolved with a
probability law of covariance $\Sigma-\varepsilon I_d\succeq0$.
Uniform convergence implies $L^1$ convergence, so the limits lie in the total-variation closure of the finite family. By Theorem~\ref{thm:completionclassification} they are densities of completion points, and its homeomorphism turns $L^1$ convergence into $d_{\mathrm W}$ convergence. This gives compactness.

At a surviving frequency $k\in L$, passing to the limit in
\eqref{eq:spectralenvelope} gives $Q(k)\ge\varepsilon|k|^2$.
Homogeneity and rational approximation extend this inequality to the
real span $V=\operatorname{span}_{\R}L$. Conversely, let
\eqref{eq:spectralfloorclosure} hold. Extend the nonnegative quadratic
form $Q-\varepsilon|\cdot|^2$ on $V$ to a nonnegative form $B$ on
$\R^d$ by setting it to zero on $V^\perp$ and taking zero cross terms.
With $\Pi_{V^\perp}$ the Euclidean orthogonal projection, set
\[
 \Sigma_t=\varepsilon I_d+B+t\Pi_{V^\perp}.
\]
These covariances satisfy the spectral floor. Since $L$ is primitive,
$V\cap\Z^d=L$. Extend $\chi$ to a torus mean $\mu$; then the Fourier
moments of $(\mu,\Sigma_t)$ tend to those of $(L,\chi,Q)$: for $k\in L$ one has $k^\top\Sigma_tk=\varepsilon|k|^2+B(k)=Q(k)$, while for integer $k\notin L$ the projection $\Pi_{V^\perp}k$ is nonzero, with norm bounded below by a constant $c_L>0$ because $V$ is rational, so $k^\top\Sigma_tk\ge tc_L^2\to\infty$ uniformly in such $k$.
The common envelope upgrades this convergence to uniform density
convergence, and the completion theorem upgrades it to Fisher convergence.
This proves the asserted characterization, including the ranks zero and $d$.

For an integer cutoff $M\ge1$, put
$\Lambda_M=\{k\in\Z^d:\|k\|_\infty\le M\}$ and
\[
 \widehat m_k=\frac1n\sum_{j=1}^n e^{ik^\top Y_j},\qquad
 G_{n,M}(\theta)=\sum_{k\in\Lambda_M}\widehat m_k e^{-ik^\top\theta},
 \qquad
 \widetilde F_{n,M}=\frac{(G_{n,M})_+}{\int(G_{n,M})_+\,\dd\nu}.
\]
The conjugate symmetry makes $G_{n,M}$ real, and its integral is one,
so the denominator is at least one. For any probability density $F$,
clipping and normalizing give
$\|\widetilde F_{n,M}-F\|_1\le2\|G_{n,M}-F\|_1$.
Parseval's identity, the bound
$\E|\widehat m_k-m_k|^2\le1/n$, and
\eqref{eq:spectralenvelope} yield the finite bound
\begin{equation}\label{eq:fourierestimatorbound}
 \sup_{P\in\mathcal M_d(\varepsilon)}
 \E\|\widetilde F_{n,M}-F_P\|_1
 \le2\left\{
 \frac{(2M+1)^d-1}{n}
 +\sum_{\|k\|_\infty>M}e^{-\varepsilon|k|^2}
 \right\}^{1/2}.
\end{equation}
For $A_\varepsilon=\sum_{j\in\Z}e^{-\varepsilon j^2}$, a union bound
on the coordinates and the Gaussian integral give
\[
 \sum_{\|k\|_\infty>M}e^{-\varepsilon|k|^2}
 \le dA_\varepsilon^{d-1}
       \frac{e^{-\varepsilon M^2}}{\varepsilon M}.
\]
Taking $M=\max\{1,\lceil\sqrt{\log(n+1)/\varepsilon}\rceil\}$
proves \eqref{eq:spectralfloorrate}.

To obtain an estimator in the model, choose a finite $n^{-1}$-net of
$\mathcal K_\varepsilon$ in $L^1$ with centers in
$\mathcal M_d(\varepsilon)$. Choose the first center minimizing its
$L^1$ distance to $\widetilde F_n$; this is a measurable finite-choice
rule. The triangle inequality gives
\[
 \|F_{\widehat P_n}-F_P\|_1
 \le2\|\widetilde F_n-F_P\|_1+n^{-1}.
\]
Finally, on the compact space $\mathcal K_\varepsilon$, the inverse
density map is uniformly continuous and $d_{\mathrm W}$ has finite
diameter. For each $\eta>0$ choose $\delta>0$ such that $L^1$ distance
below $\delta$ implies Fisher distance below $\eta$. Markov's inequality
and the last bound show that the supremum of the expected Fisher loss
is at most $\eta+\operatorname{diam}(\mathcal K_\varepsilon)
\delta^{-1}\sup_P\E\|F_{\widehat P_n}-F_P\|_1$.
First send $n\to\infty$ and then $\eta\downarrow0$ to obtain
\eqref{eq:spectralfloorfisher}.
\end{proof}

\begin{remark}[What the shear obstruction requires]
\label{rem:spectralfloorscope}
For a fixed amplitude parameter $a$, the two-dimensional shear
covariances in \eqref{eq:shearpoints} satisfy
$\lambda_{\min}(\Sigma_k)\sim a/k^2$ and
$\lambda_{\max}(\Sigma_k)\sim ak^2$. They therefore leave every class
$\mathcal M_d(\varepsilon)$. If $a$ is increased to maintain a fixed
spectral floor, their Fourier separation scale $e^{-a/2}$ tends to zero;
the fixed-scale infinite packing used in the minimax lower bound is lost.
An unbounded condition number alone is not an obstruction: for example,
$\operatorname{diag}(\varepsilon,t,\varepsilon,\ldots,\varepsilon)$
has unbounded condition number as $t\to\infty$, but belongs to the
uniformly estimable class of Theorem~\ref{thm:spectralfloor}.
The lower bound selects a fixed flat background geometry on the torus;
it is not invariant under arbitrary integer shears. The theorem does
not require, or assert, that a spectral floor is necessary for uniform
estimation on every restricted covariance family.
\end{remark}

\begin{proposition}[A lower scale for the regularized minimax risk]
\label{prop:spectralfloorlower}
For $\varepsilon>0$, $d\ge1$, and $n\ge1$,
\[
 \inf_{\widehat F_n}\sup_{P\in\mathcal M_d(\varepsilon)}
 \E_{P^{\otimes n}}\|\widehat F_n-F_P\|_1
 \ge \frac{3e^{-\varepsilon/2}\min\{1,\sqrt\varepsilon\}}
 {8\pi\sqrt n}.
\]
\end{proposition}
\begin{proof}
Use the two wrapped laws with covariance $\varepsilon I_d$ and means
$0$ and $h_ne_1$, where
$h_n=\min\{1,\sqrt\varepsilon\}/(2\sqrt n)$.
Their first trigonometric moments give $L^1$ separation
\[
 s_n\ge e^{-\varepsilon/2}|e^{ih_n}-1|
 \ge (2/\pi)e^{-\varepsilon/2}h_n.
\]
The unwrapped product Gaussian laws have affinity
$\exp\{-nh_n^2/(8\varepsilon)\}$.
By Cauchy--Schwarz, their total variation is at most
$\sqrt{1-e^{-nh_n^2/(4\varepsilon)}}\le1/4$; applying the wrapping
map cannot increase total variation. If $p_0,p_1$ are the wrapped
product densities, the triangle inequality for each estimated density
gives a sum of the two risks at least
$s_n\int\min(p_0,p_1)=s_n(1-d_{\mathrm{TV}}(p_0,p_1))$.
Their maximum is consequently at least $3s_n/8$, which is the stated
bound.
\end{proof}

Thus the regularized minimax $L^1$ risk lies between a constant times
$n^{-1/2}$ and a constant times $n^{-1/2}\{\log(n+1)\}^{d/4}$.
The logarithmic gap is not resolved here.

\subsection{The exact compactness criterion}
The spectral floor is one way to exclude the shear mechanism. The next result identifies exactly which subfamilies have compact closure.

\begin{proposition}[Compact subfamilies]\label{prop:compactcriterion}
Let $\mathcal F\subset\mathcal M_d$. The closure of $\mathcal F$ in $\widehat{\mathcal M}_{d,\mathrm W}$ is compact if and only if
\begin{equation}\label{eq:uniformdecay}
 h_{\mathcal F}(n):=\inf_{(\mu,\Sigma)\in\mathcal F}n^\top\Sigma n\longrightarrow\infty
 \qquad(|n|\to\infty),
\end{equation}
equivalently $\sup_{P\in\mathcal F}|m_n(F_P)|=e^{-h_{\mathcal F}(n)/2}\to0$. No condition on the means is needed.
\end{proposition}
\begin{proof}
Suppose first that the closure is compact but \eqref{eq:uniformdecay} fails. Then there are $c>0$, frequencies $n_j$ with $|n_j|\to\infty$, and $P_j\in\mathcal F$ with $|m_{n_j}(F_{P_j})|\ge c$. Passing to a subsequence, $P_j\to P$ in the completion, so $\|F_{P_j}-F_P\|_1\to0$ by Lemma~\ref{lem:observables}. Hence
$|m_{n_j}(F_{P_j})-m_{n_j}(F_P)|\le\|F_{P_j}-F_P\|_1\to0$, while $m_{n_j}(F_P)\to0$ by the Riemann--Lebesgue lemma. This contradicts $|m_{n_j}(F_{P_j})|\ge c$.

Conversely, assume \eqref{eq:uniformdecay}. Every point of the closure is a limit of points of $\mathcal F$, so by a diagonal argument it suffices to show that every sequence $P_j=(\mu_j,\Sigma_j)$ in $\mathcal F$ has a subsequence converging in the completion. Put $Q_j(n)=n^\top\Sigma_jn$. By a diagonal argument over the countable set $\Z^d$, we may assume that $Q_j(n)$ converges in $[0,\infty]$ and $e^{in\cdot\mu_j}$ converges in $S^1$ for every $n$. Let $L$ be the set of $n$ with finite limit $Q(n)$, and $\chi(n)=\lim_je^{in\cdot\mu_j}$ for $n\in L$. As in the proof of Lemma~\ref{lem:densityclosure}, $L$ is a primitive subgroup, $\chi$ is a character, and the limits on a basis of $L$ and its pairwise sums define a positive-semidefinite quadratic form $Q$ on $V=\operatorname{span}_{\R}L$ with $Q(n)=\lim_jQ_j(n)$ for $n\in L$. Moreover $Q(n)\ge h_{\mathcal F}(n)$ on $L$.

The form $Q$ is positive definite on $V$. Otherwise its kernel $K\subset V$ is nonzero. For $\delta>0$ and $R>0$, the convex symmetric set of $x\in V$ whose component in $K$ has norm at most $R$ and whose component in $K^\perp\cap V$ has norm at most $\delta$ has volume tending to infinity with $R$. By Minkowski's theorem it contains a nonzero $n_\delta\in L$, and $Q(n_\delta)\le\|Q\|\delta^2$. If the $n_\delta$ are unbounded as $\delta\to0$, this contradicts $Q(n)\ge h_{\mathcal F}(n)\to\infty$. If they are bounded, some nonzero $n\in L$ has $Q(n)=0$; then $h_{\mathcal F}(kn)\le Q(kn)=k^2Q(n)=0$ for every integer $k$, again contradicting \eqref{eq:uniformdecay}.

Next, $\min_{n\notin L}Q_j(n)\to\infty$. Given $C>0$, the set $\{n:h_{\mathcal F}(n)\le C\}$ is finite by \eqref{eq:uniformdecay}. For the finitely many $n\notin L$ in this set, $Q_j(n)\to\infty$ by the definition of $L$; every other $n\notin L$ has $Q_j(n)\ge h_{\mathcal F}(n)>C$. Hence $\min_{n\notin L}Q_j(n)>C$ for all large $j$.

Thus the characters and quadratic forms restricted to $L$ converge to $\chi$ and $Q$, $Q$ is positive definite, and \eqref{eq:outsidecriterion} holds. By Theorem~\ref{thm:completionclassification}, $P_j$ converges to the point $(L,\chi,Q)$ of the completion; if $L=\Z^d$ this point lies in $\mathcal M_d$.
\end{proof}

For $\mathcal M_d(\varepsilon)$ one has $h(n)\ge\varepsilon|n|^2$, so Proposition~\ref{prop:compactcriterion} recovers the compactness in Theorem~\ref{thm:spectralfloor}. The criterion is strictly weaker. For the sequence $\Sigma_j=jB_j$ of Example~\ref{ex:d2}, every nonzero $n$ satisfies $n^\top\Sigma_jn\ge j$, and for fixed $j$ the form $n^\top\Sigma_jn$ grows quadratically in $n$; hence $\inf_jn^\top\Sigma_jn\to\infty$ as $|n|\to\infty$, although $\lambda_{\min}(\Sigma_j)\to0$. In contrast, the shear sequence at $E=U_d$ in the proof of Theorem~\ref{thm:noncompact} has $|m_{(1,k)}(P_k)|=q$ for every $k$, so it violates \eqref{eq:uniformdecay}, as it must.

Compactness also suffices for uniform estimation. We use the minimum-distance estimate of Yatracos \cite{Yatracos1985}; see also \cite[Chapter~6]{DevroyeLugosi2001}. For completeness we include the short argument.

\begin{lemma}[Minimum-distance selection]\label{lem:yatracos}
Let $f_1,\dots,f_N$ be probability densities on $\T^d$, let $\mathcal A$ be the family of sets $\{f_i>f_k\}$, $i\ne k$, and let $\mathbb P_n$ be the empirical measure of $n$ independent observations from a density $f$. Choose $\widehat\imath$ minimizing $\Delta_i=\max_{A\in\mathcal A}|\int_Af_i\dd\nu-\mathbb P_n(A)|$, with ties broken by the smallest index. Then
\[
 \E\|f_{\widehat\imath}-f\|_1\le3\min_i\|f_i-f\|_1+4\sqrt{\frac{\log(2N^2)}{2n}}.
\]
\end{lemma}
\begin{proof}
Put $\Delta=\max_{A\in\mathcal A}|\mathbb P_n(A)-\int_Af\dd\nu|$. For every $i$, $\Delta_i\le\frac12\|f_i-f\|_1+\Delta$. If $\widehat\imath=i$ there is nothing to prove in the next step; otherwise, with $A=\{f_{\widehat\imath}>f_i\}\in\mathcal A$, one has $\|f_{\widehat\imath}-f_i\|_1=2\int_A(f_{\widehat\imath}-f_i)\dd\nu\le2(\Delta_{\widehat\imath}+\Delta_i)\le4\Delta_i$. Hence $\|f_{\widehat\imath}-f\|_1\le3\|f_i-f\|_1+4\Delta$ for every $i$. Each $\mathbb P_n(A)-\int_Af\dd\nu$ is an average of $n$ independent centered variables with values in an interval of length one, hence sub-Gaussian with variance proxy $1/(4n)$ by Hoeffding's lemma. The maximal inequality over the at most $2N^2$ signed variables gives $\E\Delta\le\sqrt{\log(2N^2)/(2n)}$.
\end{proof}

\begin{corollary}[Uniform consistency on compact subfamilies]\label{cor:compactconsistency}
If $\mathcal F\subset\mathcal M_d$ satisfies \eqref{eq:uniformdecay}, there are estimators $\widehat P_n$ with values in $\mathcal F$ such that
\[
 \sup_{P\in\mathcal F}\E_{P^{\otimes n}}\|F_{\widehat P_n}-F_P\|_1\longrightarrow0
 \quad\text{and}\quad
 \sup_{P\in\mathcal F}\E_{P^{\otimes n}}d_{\mathrm W}(\widehat P_n,P)\longrightarrow0 .
\]
Quantitatively, let $N_{\mathcal F}(\delta)$ be the smallest number of points of $\mathcal F$ whose densities form a $\delta$-net of $\{F_P:P\in\mathcal F\}$ in $L^1$. Then the $L^1$ risk of these estimators satisfies
\begin{equation}\label{eq:entropyrate}
 \sup_{P\in\mathcal F}\E_{P^{\otimes n}}\|F_{\widehat P_n}-F_P\|_1
 \le\inf_{\delta>0}\left\{3\delta+4\sqrt{\frac{\log2+2\log N_{\mathcal F}(\delta)}{2n}}\right\}.
\end{equation}
If moreover $T_{\mathcal F}(M)=\sum_{\|k\|_\infty>M}e^{-h_{\mathcal F}(k)}$ is finite, the density estimator $\widetilde F_{n,M}$ of Theorem~\ref{thm:spectralfloor} satisfies
\begin{equation}\label{eq:enveloperate}
 \sup_{P\in\mathcal F}\E_{P^{\otimes n}}\|\widetilde F_{n,M}-F_P\|_1
 \le2\left\{\frac{(2M+1)^d-1}{n}+T_{\mathcal F}(M)\right\}^{1/2}.
\end{equation}
\end{corollary}
\begin{proof}
Let $\mathcal K$ be the compact closure of $\mathcal F$ given by Proposition~\ref{prop:compactcriterion}. Its density image is compact in $L^1$, and $\mathcal F$ is dense in $\mathcal K$. Hence for each $\delta>0$ there are finitely many $P_1,\dots,P_{N(\delta)}\in\mathcal F$ whose densities form a $\delta$-net of the densities of $\mathcal F$ in $L^1$. Choose $\delta_n\downarrow0$ so slowly that $\log N(\delta_n)/n\to0$, and let $\widehat P_n$ be the point selected by Lemma~\ref{lem:yatracos} from this net. Then $\sup_{P\in\mathcal F}\E\|F_{\widehat P_n}-F_P\|_1\le3\delta_n+4\{\log(2N(\delta_n)^2)/(2n)\}^{1/2}\to0$. Choosing instead $\delta$ to minimize the right side at each $n$ gives \eqref{eq:entropyrate}. The estimator is measurable, since it is a finite selection by comparison of empirical frequencies.

For the Fisher loss, the density map is a homeomorphism from $\mathcal K$ onto a compact set of densities, so its inverse is uniformly continuous. Given $\eta>0$, choose $\gamma>0$ such that $\|F_{P'}-F_P\|_1<\gamma$ implies $d_{\mathrm W}(P',P)<\eta$ on $\mathcal K$. With $D$ the finite diameter of $\mathcal K$, Markov's inequality gives $\E d_{\mathrm W}(\widehat P_n,P)\le\eta+D\,\gamma^{-1}\E\|F_{\widehat P_n}-F_P\|_1$ uniformly in $P\in\mathcal F$; its limit superior is at most $\eta$, and $\eta$ is arbitrary.

Finally, the proof of \eqref{eq:fourierestimatorbound} used only $\E|\widehat m_k-m_k|^2\le1/n$, Parseval's identity, and a bound for $\sum_{\|k\|_\infty>M}|m_k(F_P)|^2$. Since $|m_k(F_P)|^2\le e^{-h_{\mathcal F}(k)}$ for $P\in\mathcal F$, the same argument gives \eqref{eq:enveloperate}.
\end{proof}

The envelope bound applies beyond spectral floors. For the family $\mathcal F=\{(\mu,jB_j):\mu\in\T^2,\ j\ge1\}$ of Example~\ref{ex:d2}, whose smallest eigenvalues tend to zero, one has $h_{\mathcal F}(n)\ge\|n\|_\infty$. Indeed, with $x=jn_1-n_2$, $n^\top(jB_j)n=jn_1^2+jx^2\ge\max\{|n_1|,\,j|n_1|+|x|\}\ge\|n\|_\infty$ when $n_1\ne0$, and it equals $jn_2^2\ge|n_2|$ when $n_1=0$. Hence $T_{\mathcal F}(M)\le\sum_{m>M}8m\,e^{-m}\le C(M+1)e^{-M}$, and the choice $M=\lceil\log(n+1)\rceil$ in \eqref{eq:enveloperate} gives
\[
 \sup_{P\in\mathcal F}\E\|\widetilde F_{n,M}-F_P\|_1\le C\,\frac{\log(n+1)}{\sqrt n},
\]
a rate of the same order as under a spectral floor, although this family has none.

Proposition~\ref{prop:compactcriterion} and Corollary~\ref{cor:nouniformestimation} together describe the dichotomy. On every subfamily satisfying \eqref{eq:uniformdecay}, uniform estimation is possible. Every fixed neighborhood of an added point violates \eqref{eq:uniformdecay}, and there it is impossible. Whether every subfamily violating \eqref{eq:uniformdecay} forbids uniform estimation is not settled here; the proof of Corollary~\ref{cor:nouniformestimation} uses the additional likelihood-ratio control supplied by integer shears.

\subsection{A numerical illustration}\label{subsec:simulation}
Figure~\ref{fig:simulation} illustrates the dichotomy for $d=2$ with the truncated Fourier estimator $\widetilde F_{n,M}$ of Theorem~\ref{thm:spectralfloor}, using the same cutoff $M_n=\lceil\{\log(n+1)/0.5\}^{1/2}\rceil$ in both cases. For each sample size we estimate, by twelve Monte Carlo repetitions, the mean $L^1$ error at every member of a finite subfamily and report the largest of these means. The first subfamily consists of the shear laws $\Sigma_k=2\left(\begin{smallmatrix}1+k^2&-k\\-k&1\end{smallmatrix}\right)$, $k\in\{0,1,2,3,5,8,13,21,34\}$, all at the same Fisher distance from $U_2$. The second consists of twelve distinct covariances of the forms $0.5I_2+tvv^\top$ and $0.5I_2+B_k$ satisfying $\Sigma\succeq0.5I_2$, including large condition numbers. Over $100\le n\le30000$, the worst-case error on the shear family levels off near $0.47$, because every member with $k>M_n$ retains a moment of size $e^{-1}$ at the frequency $(1,k)$, which the estimator does not see. On the spectral-floor family it decreases with fitted log-log slope $-0.44$, consistent with $n^{-1/2}$ up to the logarithmic factor in \eqref{eq:spectralfloorrate}. The figure is an illustration with one estimator and finite subfamilies; the lower bound of Corollary~\ref{cor:nouniformestimation} applies to every estimator and is proved analytically.

\begin{figure}[ht]
\centering
\includegraphics[width=.82\textwidth]{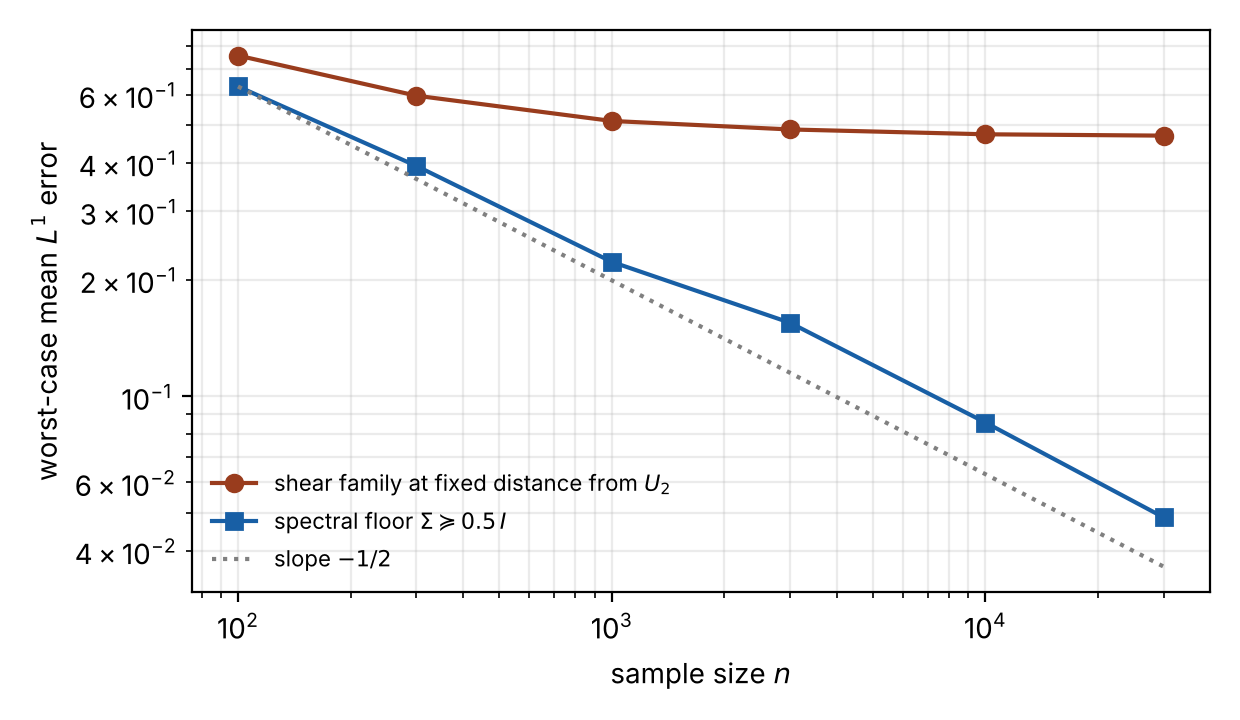}
\caption{Worst-case Monte Carlo $L^1$ error of the truncated Fourier estimator, over the shear family at a common distance from $U_2$ and over a spectral-floor family $\Sigma\succeq0.5I_2$, for $d=2$. Twelve repetitions per law and sample size; the dotted line has slope $-1/2$. Computed by the ancillary script \texttt{simulate\_estimation.py}.}\label{fig:simulation}
\end{figure}

\section{Discussion}
The main geometric and statistical conclusions are consequences of the same surviving-frequency structure. Unrestricted shears retain separated Fourier signals inside arbitrarily small Fisher neighborhoods, producing a nonlocally compact completion and a sample-size-independent local minimax obstruction. A fixed spectral floor gives a common summable Fourier envelope, compact closure in the same completed metric, and uniform estimation. The condition number may still diverge in that regularized class; loss of the spectral floor and movement of fixed-amplitude signals to high frequencies are the features used by the lower bound.

The metric results provide more than an obstruction theorem. Independent-block slices are isometric, and for product inputs the canonical projection onto a uniform-factor piece realizes the distance to that piece. Every added law lies on a minimizing geodesic; at residual rank $d-1$ this gives finite endpoints for which the original manifold has no distance-attaining curve. Compatible transverse pieces have a right-angle distance asymptotic and a shorter bridge than the path through their intersection. These statements coexist with contractible small balls and finite Hausdorff dimension.

The remaining questions have a narrower scope. General existence of minimizers between arbitrary points in the nonproper completion, exact distances between incompatible frequency pieces, and sectional-curvature limits along interior approaches to partially uniform laws remain open here. The logarithmic gap between the regularized $L^1$ risk bounds is also unresolved. These questions are distinct from the intrinsic classification, the spectral-floor recovery theorem, and the explicit minimizing curves proved above.

The setting can be enlarged in several directions that we do not pursue. First, when only some coordinates are angular, observations lie in $\R^k\times\T^m$, as for wind direction together with speed; the angular frequencies should again be classified by primitive subgroups, now coupled to the linear coordinates through conditional covariances. Second, other wrapped families whose characteristic functions have the form $e^{in\cdot\mu}\varphi(n^\top\Sigma n)$ with $\varphi$ positive and strictly decreasing, such as wrapped multivariate stable laws with $\varphi(t)=e^{-t^{\alpha/2}}$, have the same integer-shear symmetry and are expected to admit the same lattice description of their total-variation closures, but their Fisher distances require separate estimates. For general elliptical generators this fails, since $\varphi$ may vanish. Third, identifying laws related by integer torus automorphisms identifies the separated shear sequences; the completion of the resulting quotient, and its relation to the geometry of spaces of lattices, is a different and natural problem.

\section*{Declarations}\phantomsection\label{sec:declarations}

\subsection*{Funding}
The authors received no specific funding for this work.

\subsection*{Competing interests}
The authors declare no competing interests.

\subsection*{Data and code availability}
No data sets were generated or analyzed. The ancillary files contain scripts with run records that corroborate selected formulas numerically: the heat contraction in one and two observation dimensions, the sharp distance coefficient on fixed covariance rays, the theta identity in the proof of Corollary~\ref{cor:nouniformestimation}, the bounds of Theorem~\ref{thm:spectralfloor}, and the distance bounds between pieces; a further script produces the Monte Carlo illustration in Figure~\ref{fig:simulation}. These checks are not interval certificates, and no proof in the paper depends on them.

\subsection*{Use of generative AI}
During the preparation of this work the authors used generative AI tools for proof exploration, drafting and revision; for independent consistency checks of the proofs; for code for the numerical checks, the figures, and the verification scripts; and for language editing. After using these tools, the authors reviewed and edited the content as needed. They checked all mathematical content and take full responsibility for the publication.

\subsection*{Acknowledgements}
This research was initiated during the Honors Summer Math Camp (HSMC), organized by Mathworks at Texas State University. We thank Juna Lee, Nathan Negera, Kalea Plotnick, and Antony Wang for their participation in a preliminary study of the circular model.

\printbibliography
\end{document}